\documentclass[11pt]{article}
\usepackage[utf8]{inputenc}
\usepackage{hyperref}
\usepackage{amsmath}
\usepackage{amsfonts}
\usepackage{amssymb}
\usepackage{amsthm}
\usepackage{graphicx}
\usepackage{subcaption}
\usepackage[text={17cm,23cm},centering]{geometry}
\usepackage{cite}

\usepackage{mathrsfs}
\usepackage{mathtools}
\usepackage{graphicx}
\usepackage{enumitem}
\usepackage{url}

\usepackage{color}
\usepackage[dvipsnames]{xcolor}

\newtheorem{thm}{Theorem}
\newtheorem{lem}[thm]{Lemma}
\newtheorem{prop}[thm]{Proposition}
\newtheorem{coro}[thm]{Corollary}
\newtheorem*{thm*}{Theorem}
\theoremstyle{definition}
\newtheorem{defn}[thm]{Definition}

\theoremstyle{remark}
\newtheorem{rmk}[thm]{Remark}

\DeclareMathOperator{\reg}{Reg}

\DeclareMathOperator{\supp}{supp}
\DeclareMathOperator{\diam}{diam}
\DeclareMathOperator{\dist}{dist}
\DeclareMathOperator{\divergence}{div}

\DeclareMathOperator{\crit}{crit}
\DeclareMathOperator{\lip}{Lip}

\DeclareMathOperator{\acc}{acc}
\DeclareMathOperator{\essacc}{ess\,acc}

\newcommand{\R}{\mathbb{R}}
\newcommand{\N}{\mathbb{N}}

\newcommand{\lebesgue}{\mathsf{Leb}}
\newcommand{\lipschitz}{\mathsf{Lip}}

\renewcommand{\geq}{\geqslant}
\renewcommand{\leq}{\leqslant}
\newcommand{\meas}[2]{\mu_{#1|_{#2}}}

\title{Long term dynamics of the subgradient method for Lipschitz path differentiable functions\\
}
\author{
 J\'er\^ome Bolte, Edouard Pauwels, and
 Rodolfo R\'ios-Zertuche
}

\begin{document}
\maketitle
\begin{abstract}
We consider the long-term dynamics of the vanishing stepsize subgradient method in the case when the objective function is neither smooth nor convex. We assume that this function is locally Lipschitz and path differentiable, i.e., admits a chain rule. Our study departs from other works in the sense that we focus on the behavoir of the oscillations, and to do this we use closed measures. We recover known convergence results, establish new ones, and show a local principle of oscillation compensation for the velocities. Roughly speaking, the time average of gradients {\em around} one limit point vanishes. 
This allows us to further analyze  the structure of oscillations, and establish their perpendicularity to the general drift.
\end{abstract}
\newpage
\tableofcontents
\newpage
\section{Introduction}
\label{sec:intro}

 The predominance of huge scale complex nonsmooth nonconvex problems in the development of certain artificial intelligence methods, has brought back rudimentary, numerically cheap, robust methods, such as subgradient algorithms, to the forefront of contemporary numerics, see e.g., \cite{higham2019deep,duchi2011adaptive,kingma2014adam,barakat2019convergence,bianchi2020convergence}. We investigate here some of the properties of the archetypical algorithm within this class, namely, the vanishing stepsize  subgradient method of Shor. Given $f:\R^n\to\R$ locally Lipschitz, it reads
 $$x_{i+1}\in x_i-\varepsilon_i\partial^c f(x_i), \:\: x_0\in \R^n,$$
 where $\partial^c f$ is the Clarke subgradient, $\varepsilon_i \to 0$, and 
 $\sum_{i=0}^\infty\varepsilon_i=\infty$. 
  This dynamics, illustrated in Figure \ref{fig:exSubgradDescentBanana}, has its roots in Cauchy's gradient method and seems to originate in Shor's thesis \cite{shorthesis}. The idea is natural at first sight: one accumulates small subgradient steps  to make good progress on average while  hoping that oscillations will be tempered by the vanishing steps.  For the convex case, the theory was developed by Ermol'ev \cite{ermolev66}, Poljak \cite{poljak67},  Ermol'ev--Shor \cite{ermol67}. It is a quite mature theory, see e.g. \cite{nemirovskii1979complexity,nesterov2013introductory}, which still has a considerable success through the famous mirror descent of Nemirovskii--Yudin \cite{nemirovskii1979complexity,beck2003mirror} and its endless variants. In the nonconvex case, developments of more sophisticated methods were made, see e.g. \cite{kiwiel2001,hare2010redistributed,noll2013bundle}, yet little  was known for the raw method until recently.

The work of  Davis et al. \cite{Davis2019}, see also \cite{bianchi2019constant}, revolving around the fundamental paper of Bena\"im--Hofbauer--Sorin \cite{BHS}, brought the first breakthroughs. It relies on a classical idea of Euler: small-step discrete dynamics resemble their continuous counterparts. As established by Ljung \cite{ljung1977analysis}, this observation can be made rigorous for large times in the presence of good Lyapunov functions. Bena\"im--Hofbauer--Sorin \cite{BHS}  showed further that the transfer of asymptotic properties from continuous differential inclusions to small-step discrete methods is valid under rather weak compactness and dissipativity assumptions. This general result, combined with features specific to the subgradient case, allowed to establish several optimization results such as the convergence to the set of critical points, the convergence in value, convergence in the long run in the presence of noise \cite{Davis2019,adil,boltepauwels,bianchi2020convergence}.

Usual properties expected from an algorithm are diverse: convergence of iterates, convergence in values, rates, quality of optimality, complexity, or prevalence of minimizers. Although in our setting some aspects seem hopeless without strong assumptions,  most of them remain largely unexplored.   Numerical successes suggest however that the apparently  erratic process of subgradient dynamics has  appealing stability properties beyond the already delicate subsequential convergence to critical points. %

In order to address some of these issues, this paper avoids the use of the theory of \cite{BHS} and focuses on the delicate question of oscillations, which is illustrated on Figures \ref{fig:exSubgradDescentBanana} and \ref{fig:convex}.
\begin{figure}[!ht]
    \centering
    \includegraphics[width=.6\textwidth]%
    {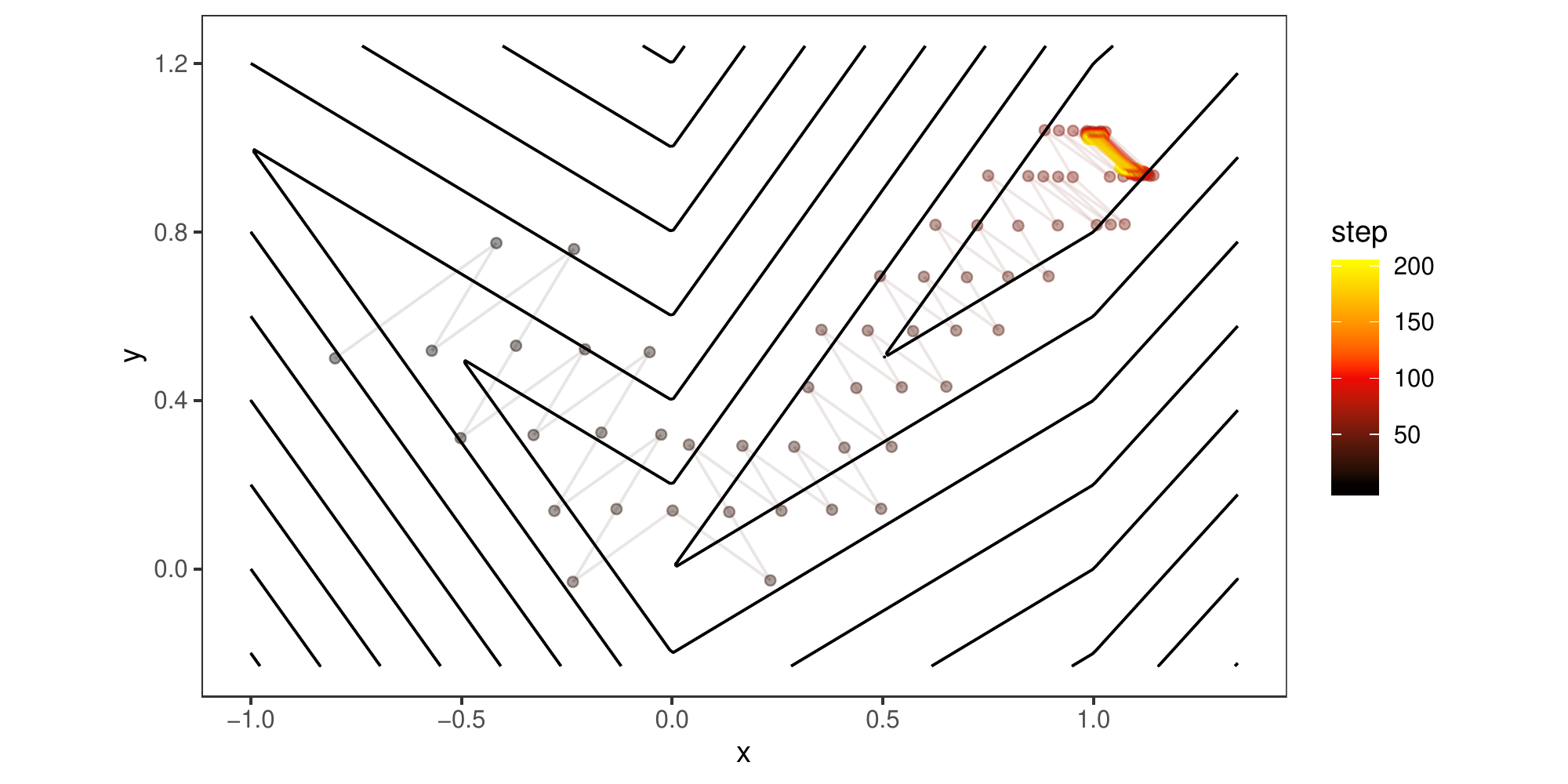}
    \caption{Contour plot of a Lipschitz function with a subgradient sequence. The color reflects the iteration count. The sequence converges to the unique global minimum, but is constantly oscillating.}
    \label{fig:exSubgradDescentBanana}
\end{figure}

\begin{figure}[!ht]
\centering
 \includegraphics[width=\linewidth]{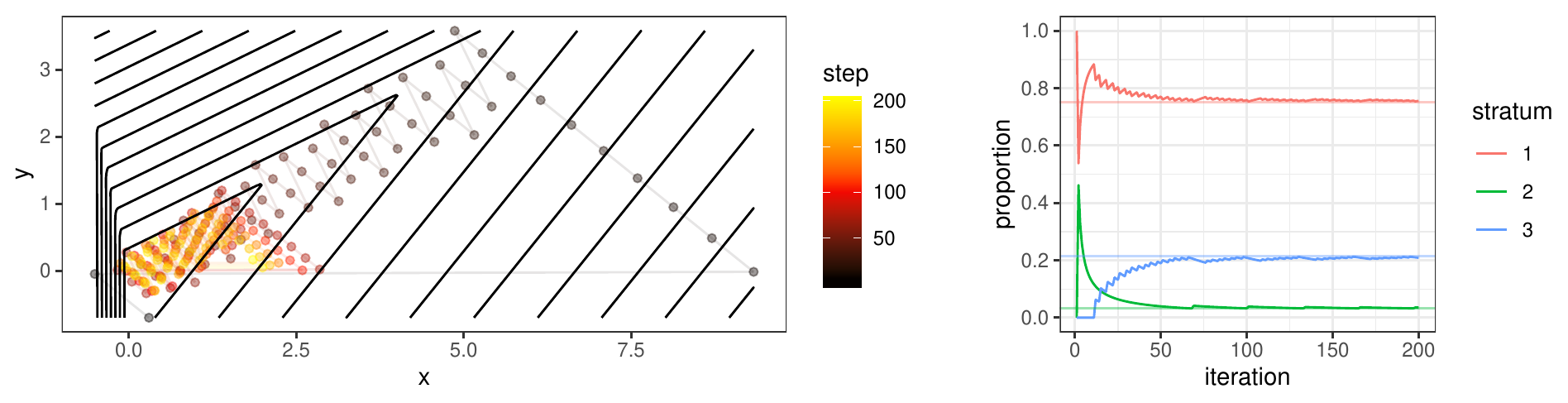}
\caption{On the left,  the contour plot of a convex polyhedral function with three strata, where the gradient is constant.
A subgradient sequence starts at $(0.3, -0.7)$ and converges to the origin with  an apparent erratic behavior. On the right, we discover that the behavior is not completely erratic. The oscillation compensation phenomenon contributes some structure: 
the proportions $\lambda_i$ of time spent in each region where the function has constant gradient $g_i$, $i=1,2,3$, converge so that we have precisely $\lambda_1g_1+\lambda_2g_2+\lambda_3g_3=0$.
}
\label{fig:convex}
\end{figure}

In general, as long as the sequence $\{x_i\}_i$ remains bounded, we always have
\begin{equation}\label{eq:globalcomp}
 \frac{x_N-x_0}{\sum_{i=0}^N\varepsilon_i}=\frac{\sum_{i=0}^N\varepsilon_iv_i}{\sum_{i=0}^N\varepsilon_i}\to 0, \mbox{ where $v_i\in \partial^cf(x_i)$}.
\end{equation}
This fact, that could be called ``global oscillation compensation,'' does not prevent the trajectory to oscillate fast around a limit cycle, as illustrated in \cite{daniilidisdrusvyatskiy}, and is therefore unsatisfying from the stabilization perspective of minimization. The phenomenon \eqref{eq:globalcomp} remains true even when $\{x_i\}_i$ is not a gradient sequence, as in the case of discrete game theoretical dynamical systems \cite{BHS}.

In this work, we adapt the theory of closed measures, which was originally developed in the calculus of variations (see for example \cite{bangert1999minimal,patrick}), to the study of discrete dynamics. Using it, we establish several local oscillation compensation results for path differentiable functions. Morally, our results in this direction say that for limit points $x$ we have
\begin{equation}\label{eq:intuitivecompensation}
\textrm{``}\,\lim_{\substack{\delta\searrow0\\N\to+\infty}}\frac{\displaystyle\sum_{\substack{0\leq i\leq N\\ \|x-x_i\|\leq \delta}}\varepsilon_iv_i}{\displaystyle\sum_{\substack{0\leq i\leq N\\ \|x-x_i\|\leq \delta}}\varepsilon_i}=0\,\textrm{''}
\end{equation}
See Theorems \ref{thm:main} and \ref{thm:addendum} for precise statements, and a discussion in Section \ref{sec:interpretations}.%

While this does not imply the convergence of $\{x_i\}_i$, it does mean that the drift emanating from the average velocity of the sequence  vanishes as time elapses. This is made  more explicit in the parts of those theorems that show that, given two limit points $x$ and $y$ of the sequence $\{x_i\}_i$, the time it takes for the sequence to flow from a small ball around $x$ to a small ball around $y$ must eventually grow infinitely long, so that the overall speed of the sequence as it traverses the accumulation set becomes extremely slow.

With these types of results, we evidence new phenomena:
\begin{itemize}
\item[---] while the sequence may not converge, it will spend most of the time oscillating near the critical set of the objective function, and it appears that there are persistent accumulation points whose importance is predominant;
\item[---] under weak Sard assumptions, we recover the convergence results of \cite{Davis2019} and improve them by  oscillation compensations results,
\item[---] oscillation structures itself orthogonally to the limit set, so that the incremental drift along this set is negligible with respect to the time increment $\varepsilon_i$.
\end{itemize}

These results are made possible by the use of closed measures. These measures capture the accumulation behavior of the sequence $\{x_i\}_i$ along with the ``velocities'' $\{v_i\}_i$. The simple idea of not throwing away the information of the vectors $v_i$ allows one to recover a lot of structure in the limit, that can be interpreted as a portrait of the long-term behavior of the sequence. The theory that we develop in Section \ref{sec:closedmeas} should apply to the analysis of the more general case of small-step algorithms. Along the way, for example, we are able to establish a new connection between the discrete and continuous gradient flows (Corollary \ref{cor:discretecontinuous}) that complements the point of view of \cite{BHS}.

\paragraph{Notations and organization of the paper.} Let $n$ be a positive integer, and $\R^n$ denote $n$-dimensional Euclidean space. The space $\R^n\times\R^n$ of couples $(x,v)$ is seen as the phase space consisting of positions $x\in \R^n$ and velocities $v\in \R^n$. For two vectors $u=(u_1,\dots,u_n)$ and $v=(v_1,\dots,v_n)$, we let $u\cdot v=\sum_{i=1}^nu_iv_i$. The norm $\|v\|=\sqrt{v\cdot v}$ induces the distance $\dist(x,y)=\|x-y\|$, and similarly on $\R^n\times\R^n$. The Euclidean gradient of $f$ is denoted by $\nabla f(x)$. The set $\N$ contains all the nonnegative integers.%

In Section \ref{sec:algorithmframework} we give the definitions necessary to state our results, which we do in Section \ref{sec:results}. 
The proofs of our results will be given in Section \ref{sec:proofs}. Before we broach those arguments, we need to develop some preliminaries regarding our main tool, the so-called closed measures; we do this in Section \ref{sec:prelims}. %

\section{Algorithm and framework}
\label{sec:algorithmframework}

\subsection{The vanishing step subgradient method}
\label{sec:smallstepsubgrad}
Consider a locally Lipschitz functions $f:\R^n\to\R$, denote by $\reg f$ the set of its differentiability points which is dense by Rademacher's theorem (see for example \cite[Theorem 3.2]{evans2015measure}). The  Clarke subdifferential of $f$ is defined by 
\begin{equation*}
\partial^cf(x)=\mbox{conv}\,\big\{v\in \R^n:\; \textrm{there is a sequence $ \{y_k\}_k\subset \reg f$ with $y_k\to x$ and $\nabla f(y_k)\to v$}  \big\}    
\end{equation*}
where $\mbox{conv}\,S$ denotes the closed convex envelope of a set $S\subset\R^n$; see \cite{clarke}.

A point $x$ such that $0\in\partial^cf(x)$, is called {\em critical}. The \emph{critical set}
is \[\crit f=\{x\in\R^n:0\in\partial^cf(x)\}. \]
It contains local minima and maxima.

The algorithm of interest in this work is:

\begin{defn}[Small step subgradient method] \label{def:subgradient}
Let $f\colon\R^n\to\R$ be locally Lipschitz and $\{\varepsilon_i\}_{i \in \N}$ be a sequence of positive step sizes such that %
\begin{equation}\label{eq:divergestep}
\sum_{i=0}^\infty\varepsilon_i=+\infty\quad\textrm{and}\quad\varepsilon_i\searrow0.
\end{equation}
Given $x_0 \in \R^n$, consider the recursion, for $i\geq 0$,
 \[x_{i+1}=x_i-\varepsilon_iv_i, \qquad v_i \in \partial^cf(x_i).\]
 Here, $v_i$ is chosen freely among $\partial^cf(x_i)$. The sequence $\{x_i\}_{i \in \N}$ is called a \emph{subgradient sequence}.
\end{defn}

In what follows the sequence $\varepsilon_i$ is interpreted as a sequence of time increments, and it naturally defines a time counter through the formula:
\[t_i=\sum_{j=0}^{i}\varepsilon_j\]
so that $t_i\to \infty$ as $i\to \infty$. Given a sequence $\{x_i\}_i$ and a  subset $U\subseteq \R^n$, we set 
\[t_i(U)=\sum_{x_j\in U,\, j\leq i} \varepsilon_j,\]
which corresponds to the time spent by the sequence in $U$.

Recall that the \emph{accumulation set $\acc\{x_i\}_i$} of the sequence $\{x_i\}_i$ is the set of points $x\in\R^n$ such that, for every neighborhood $U$ of $x$, the intersection $U\cap\{x_i\}_i$ is an infinite set. Its elements are known as \emph{limit points}. 

If the sequence $\{x_i\}_i$ is bounded and comes from the subgradient method as in Definition \ref{def:subgradient}, then $\|x_i-x_{i+1}\|\to 0$ because $\varepsilon_i\to 0$ and $\partial^c f$ is locally bounded by local Lipschitz continuity of $f$, so $\acc\{x_i\}_i$ is compact and connected, see e.g., \cite{bolte2014proximal}.

\smallskip

Accumulation points are the manifestation of recurrent behaviors of the sequence but the frequency of the recurrence is ignored. In the presence of a time counter, here $\{t_i\}_i$, this persistence phenomenon  may be measured through presence duration  in the neighborhood of a recurrent point. This idea is formalized in the following definition:

\begin{defn}[Essential accumulation set]\label{def:essacc}
 Given a step size sequence $\{\varepsilon_i\}_i \subset \R_{\geq 0}$ and a subgradient sequence $\{x_i\}_i\subset\R^n$ as in Definition \ref{def:subgradient}, the \emph{essential accumulation set $\essacc\{x_i\}_i$} is the set of points $x\in\R^n$ such that, for every neighborhood $U\subseteq \R^n$ of $x$,
 \[\limsup_{N\to+\infty}\frac{\displaystyle\sum_{\substack{1\leq i\leq N\\ x_i\in U}}\varepsilon_i}{\displaystyle\sum_{1\leq i\leq N}\varepsilon_i}>0,
\quad\mbox{ that is, }\quad\limsup_{N\to+\infty}\frac{t_N(U)}{t_N}>0.\]

 Analogously, considering the increments $\{v_i\}_i \subset \R^n$, we say that the point $(x,w)$ is in the \emph{essential accumulation set $\essacc\{(x_i,v_i)\}_i$} if for every neighborhood $U\subset\R^n\times\R^n$ of $(x,w)$ satisfies
  \[\limsup_{N\to+\infty}\frac{\displaystyle\sum_{\substack{1\leq i\leq N\\ (x_i,v_i)\in U}}\varepsilon_i}{\displaystyle\sum_{1\leq i\leq N}\varepsilon_i}>0.\]
\end{defn}

As explained previously, the set $\essacc\{x_i\}_i$ encodes significantly recurrent behavior; it ignores sporadic escapades of the sequence $\{x_i\}_i$. Essential accumulation points are accumulation points but the converse is not true. If the sequence $\{x_i\}_i$ is bounded, $\essacc\{x_i\}_i$ is nonempty and compact, but not necessarily connected. %

\subsection{Regularity assumptions on the objective function}

\paragraph{Lipchitz continuity and pathologies.}   Recall that, given a locally Lipschitz function $f\colon\R^n\to\R$,  a {\em subgradient curve}  is an absolutely continuous curve satisfying,
\[\gamma'(t)\in-\partial^cf(\gamma(t)), \, \mbox{a.e. on $(0,+\infty)$ and }\gamma(0)=x_0.\]
By general results these curves exist, see e.g., \cite{BHS}  and references therein. In our context they embody the ideal behavior  we could hope from subgradient sequences.

First let us recall that  pathological Lipschitz functions are generic in the Baire sense, as established in \cite{Wang1999FineAP,borwein2001generalized}. In particular, generic $1$-Lipschitz functions $f:\R\to\R$ satisfy $\partial f\equiv [-1,1]$ everywhere on $\R$. This means that any absolutely curve $\gamma:\R\to\R$ with $\|\gamma'\|\leq 1$ is a subgradient curve of these functions, regardless of their specifics. Note that this implies that a curve may constantly remain away from the critical set.

The examples by Danillidis--Drusvyatskiy \cite{daniilidisdrusvyatskiy} make this erratic behaviour even more concrete. For instance, they provide a Lipschitz function $f\colon \R^2\to\R$ and a bounded subgradient curve $\gamma$ having the ``absurd" roller coaster property
\[(f\circ \gamma)(t)=\sin t,\quad t\in\R.\] 
Although not directly matching our framework, these examples show that we cannot hope for  satisfying convergence results under the  spineless  general assumption of Lipschitz continuity.

\paragraph{Path differentiability.}  We are thus led to consider functions avoiding pathologies. We choose to pertain to the  {\em fonctions saines}\footnote{Literally, ``healthy functions'' (as opposed to pathological) in French.} %
of Valadier \cite{valadier} (1989), rediscovered in several works, see e.g.  \cite{borwein,Davis2019,boltepauwels}. We   use the terminology of \cite{boltepauwels}.

\begin{defn}[Path differentiable functions]\label{def:pathdiff}
 A locally Lipschitz function $f\colon\R^n\to\R$ is \emph{path differentiable} if, for each Lipschitz curve $\gamma\colon\R\to\R^n$, for almost every $t\in\R$, the composition $f\circ\gamma$ is differentiable at $t$ and the derivative is given by
 \[(f\circ\gamma)'(t)=v\cdot \gamma'(t)\]
 for all $v\in\partial^cf(\gamma(t))$.
\end{defn}

In other words, all vectors in $\partial^cf(\gamma(t))$ share the same projection onto the subspace generated by $\gamma'(t)$. 
Note that the definition proposed in \cite{boltepauwels} is not limited to chain rules involving the Clarke subgradient, but it turns out to be equivalent to the a definition very much like the one we give here, with Lipschitz curves replaced by absolutely-continuous curves,
the equivalence being furnished by \cite[Corollary 2]{boltepauwels}. 
The current definition is slightly more  general than the original one \cite{boltepauwels}, that is, our class of functions contains the one discussed in \cite{boltepauwels}, because we require a condition only for Lipschitz curves, which are all absolutely continuous. %

The class of path differentiable functions is very large and includes many cases of interest, such as functions that are  semi-algebraic, tame (definable in an o-minimal structure), or Whitney stratifiable \cite{Davis2019} (in particular, models and loss functions used in machine learning, such as, for example, those occurring in neural network training with all the activation functions that have been considered in the literature), as well as functions that are %
convex, concave, %
see e.g.,  \cite{boltepauwels,scholtes2012introduction}. 

\paragraph{Whitney stratifiable functions.} Due to their ubiquity we detail here the properties of Whitney stratifiability and illustrate their utility. They  were first used in  \cite{Bolte2007} in the variational analysis context in order to establish Sard's theorem and Kurdyka-\L ojasiewicz inequality for definable functions, two properties which appears to be essential in the study of many subgradient related problems, see e.g.,  \cite{attouchboltesvaiter,bolte2014proximal}.

\begin{defn}[Whitney stratification]\label{def:stratification}
 Let $X$ be a nonempty subset of $\R^m$ and $p>0$. A \emph{$C^p$ stratification} $\mathcal X=\{X_i\}_{i\in I}$ of $X$ is a locally finite partition of $X=\bigsqcup_iX_i$ into connected submanifolds $X_i$ of $\R^m$ of class $C^p$ such that for each $i\neq j$
 \[\overline{X_i}\cap X_j\neq \emptyset \Longrightarrow X_j\subset \overline{X_i}\setminus X_i.\]
 A $C^p$ stratification $\mathcal X$ of $X$ \emph{satisfies Whitney's condition (a)} if, for each $x\in \overline{X_i}\cap X_j$, $i\neq j$, and for each sequence $\{x_k\}_k\subset X_i$ with $x_k\to x$ as $k\to+\infty$, and such that the sequence of tangent spaces $\{T_{x_k}X_i\}_k$ converges (in the usual metric topology of the Grassmanian) to a subspace $V\subset T_x\R^m$, we have that $T_xX_j\subset V$. A $C^p$ stratification is \emph{Whitney} if it satisfies Whitney's condition (a).
\end{defn}

\begin{defn}[Whitney stratifiable function]\label{def:stratifiablefunc}
 With the same notations as above,
 a function $f\colon \R^n\to\R^k$ is \emph{Whitney $C^p$-stratifiable} if there exists a Whitney $C^p$ stratification of its graph as a subset of $\R^{n+k}$. 
\end{defn}

Examples of Whitney stratifiable functions  are semialgebraic or tame functions, but much less structured functions are covered. This class covers most known finite dimensional optimization problems as for instance those met in the training of neural networks. Let us mention here that the subclass of tame functions have led to many results through the nonsmooth Kurdyka--{\L}ojasiewicz inequality, see e.g.,  \cite{attouchboltesvaiter}, while mere Whitney stratifiability combined with the Ljung-like theory developed in \cite{BHS} has also provided several interesting openings \cite{Davis2019,boltepauwels,bianchi2020convergence}.

\section{Main results: accumulation, convergence, oscillation compensation}
\label{sec:results}
We now present our main, results which rely on three types of increasingly demanding assumptions:
\begin{itemize}
\item[---] path differentiability (Section \ref{sec:pathdiff}),
\item[---] path differentiable functions with a weak Sard property (Section \ref{sec:pathdiffsard}),
\item[---] Whitney stratifiable functions (Section \ref{sec:whitney}).
\end{itemize}
Section \ref{sec:whitney} also contains a general result pertaining the structure of the oscillations.

The significance of the results is discussed in Section \ref{sec:interpretations}. The proofs are presented in Section \ref{sec:proofs}.

\subsection{Asymptotic dynamics for path differentiable functions}
\label{sec:pathdiff}

\begin{thm}[Asymptotic dynamics for path differentiable functions]\label{thm:main}
 Assume that $f\colon\R^n\to\R$ is locally Lipschitz path differentiable, and that $\{x_i\}_i$ is a sequence generated by the subgradient method (Definition \ref{def:subgradient}) that remains bounded. Then we have:
 \begin{enumerate}[label=\roman*.,ref=(\roman*)]
  \item\label{ma:longloops} (Lengthy separations) 
   Let $x$ and $y$ be two distinct points in $\essacc\{x_i\}_i$ such that $f(x)\leq f(y)$. Let $\{x_{i_k}\}_k$ be a subsequence such that $x_{i_k}\to x$ as $k\to+\infty$, and for each $k$ choose $i'_k>i_k$ such that $x_{i'_k}\to y$. Consider
   \[\bar T_k=\sum_{p=i_k}^{i'_k}\varepsilon_p.\]
   Then $T_k\to+\infty$. 

  \item\label{ma:oscillationcompcont} (Oscillation compensation) Let $\psi\colon\R^n\to[0,1]$ be a continuous function. Then for every subsequence $\{N_i\}_i\subset\N$ such that \[\liminf_{j\to+\infty}\frac{\displaystyle \sum_{i=0}^{N_j}\varepsilon_i\psi(x_i)}{\displaystyle\sum_{i=0}^{N_j}\varepsilon_i}>0,\]
  we have
  \[\lim_{j\to+\infty}\frac{\displaystyle\sum_{i=0}^{N_j}\varepsilon_iv_i\psi(x_i)}{\displaystyle\sum_{i=0}^{N_j}\varepsilon_i\psi(x_i)}=0.\]

  \item\label{ma:criticality} (Criticality) For all $x\in\essacc\{x_i\}_i$, $0\in\partial^cf(x)$. In other words, $\essacc\{x_i\}_{i\in\N}\subseteq\crit f$.
 \end{enumerate}
\end{thm}

\subsection{Asymptotic dynamics for path differentiable functions with a weak Sard pro\-perty}
\label{sec:pathdiffsard}

With slightly more stringent hypotheses, which are automatically valid for some important cases of lower or upper-$C^k$ functions \cite{barbet2016sard} (for $k$ sufficiently large), semialgebraic or tame functions \cite{Bolte2007}, we have:
\begin{thm}[Asymptotic dynamics for path differentiable functions: weak Sard case] \label{thm:addendum}
 In the setting of Theorem \ref{thm:main}, and if additionally $f$ is constant on the connected components of its critical set, then we also have:
  \begin{enumerate}[label=\roman*.,ref=(\roman*)]
  \item\label{ad:longseparation} (Lengthy separations version 2)  Let $x$ and $y$ be two distinct points in $\acc\{x_i\}_i$, $x\neq y$, and take $\delta>0$ small enough that the balls $B_\delta(x)$ and $B_\delta(y)$ are at a positive distance form each other, that is, $\|x-y\|>2\delta$. Consider the successive amounts of time it takes for the sequence to go from the ball $B_\delta(x)$ to the ball $B_\delta(y)$, namely,
  \[T_j=\inf\{\textstyle\sum_{p=i}^\ell\varepsilon_p:j\leq i<\ell,x_i\in B_\delta(x), x_\ell\in B_\delta(y)\}.\]
  Then $T_j\to+\infty$ as $j\to+\infty$.  
  \item\label{ad:longintervals} (Long intervals) Let $U,V$ be neighborhoods of $\bar{x} \in \acc\{x_i\}_i$ such that $\overline{U} \subset V$. Let $A\subset\N$ be the union $A=\bigcup_iI_i$ of the maximal intervals $I_i\subset\N$ of the form $I_i=[a_i,b_i]\cap \N$ for some $a_i<b_i$, such that $\{x_i\}_{i\in I_j}\subset U$ and $\{x_i\}_{i\in I_j}\cap V\neq\emptyset$. Then either there is some $I_j$ that is unbounded or
\[\lim_{j\to+\infty} |I_j| = \lim_{j\to+\infty}\sum_{i\in I_j}\varepsilon_i=+\infty.\]
  \item\label{ad:oscillationcomp} (Oscillation compensation version 2)   
  Let $U\subset V$ be two open sets as in item \ref{ad:longintervals}, and $A=\bigcup_iI_i$ be the corresponding union of maximal intervals. Then
     \[\lim_{N\to+\infty}\frac{\displaystyle\sum_{\substack{0\leq i\leq N\\i\in A}}\varepsilon_iv_i}{\displaystyle\sum_{\substack{0\leq i\leq N\\i\in A}}\varepsilon_i}=0. \]
  \item\label{ad:criticality} (Criticality) For all $x$ in the (traditional) accumulation set $\acc\{x_i\}_i$, $0\in\partial^cf(x)$. That is to say, $\acc\{x_i\}_i\subseteq\crit f$.
  \item\label{ad:convergencef} (Convergence of the values) The values $f(x_i)$ converge to a real number as $i\to +\infty$.
 \end{enumerate}
\end{thm}

\begin{rmk}
 Items \ref{ad:criticality} and \ref{ad:convergencef} of Theorem \ref{thm:addendum} can also be deduced from \cite[Proposition 3.27]{BHS} using a different approach. Up to our knowledge, items \ref{ad:longseparation}--\ref{ad:oscillationcomp} of Theorem \ref{thm:addendum} as well as Theorem \ref{thm:main} do not have counterparts in the optimization literature. %
\end{rmk}

\subsection{Oscillation structure and asymptotics for Whitney stratifiable functions} 
\label{sec:whitney}

The two next corollaries express that oscillations happen perpendicularly to the singular set of $f$,  whenever it makes sense. In particular, they are perpendicular to  $\essacc\{x_i\}_i$ and $\acc\{x_i\}_i$, respectively, wherever this is well defined. 

\begin{coro}[Perpendicularity of the oscillations]\label{cor:perp}
 In the setting of Theorem \ref{thm:addendum} (resp. Theorem \ref{thm:main}), let $(x,v)\in\R^n\times\R^n$ be in the accumulation set (resp. essential accumulation set) of $\{(x_i,v_i)\}_i$, and
 $\alpha\colon (-1,1)\to\R^n$ be a Lipschitz curve with the property that $\alpha(0)=x,\alpha'(0)=w$, and $t\to (f\circ \alpha)(t)$ is differentiable at $t=0$ and $(f\circ\alpha)'=v\cdot\alpha'(0)$ for all $v\in\partial^cf(\alpha(0))$. Then 
 \[w\cdot v=0,\] for all $v\in\partial^c f(x)$. In other words $w\in \left[\partial^c f(x)\right]^{\perp}$. 
\end{coro}

Stratifiable functions (cf. Definition \ref{def:stratifiablefunc}) allow to provide much more insight into  the oscillation compensation phenomenon:  we have seen that substantial oscillations, i.e., those generated by non vanishing subgradients, must be structured orthogonally to the limit point locus. Whitney rigidity then forces the following intuitive phenomenon: substantial bouncing drives the sequence to have limit points lying in the bed of V-shaped valleys formed by the graph of $f$.  

\begin{coro}[Oscillations and $V$-shaped valleys]\label{cor:whitney}
 Let $f\colon\R^n\to\R$ is a Whitney $C^{n}$ stratifiable function, and let $x$ be a point in the accumulation set of a sequence $\{x_i\}_i$ generated by the subgradient method as in Definition \ref{def:subgradient}. Assume that there is a subsequence $x_{i_j}\to x$ with
 \[\limsup_{j\to+\infty}\|v_{i_j}\|>0, \]
 then $x$ is contained in a stratum $S$ of dimension less than $n$, and if $w$ is tangent to $S$ at $x$ then
 \[\lim_{j\to+\infty}w\cdot v_{i_j}=0.\]
\end{coro}

This geometrical setting  is reminiscent of the partial smoothness assumptions of Lewis: a smooth path  lies in between the slopes of a sharp valley. While proximal-like methods end up in a finite time on the smooth locus \cite[Theorem 4.1]{hare2004identifying}, our result suggests that the explicit subgradient method keeps on bouncing, approaching the smooth part without actually attaining it. This confirms the intuition that finite identification does not occur, although oscillations eventually provide some information on active sets by their 
``orthogonality features.''
\subsection{Further discussion}
\label{sec:interpretations}

Theorems \ref{thm:main} and \ref{thm:addendum} describe the long-term dynamics of the algorithm. While Theorem \ref{thm:main} only talks about what happens close to $\essacc\{x_i\}_i$ and explains only what the most frequent persistent behavior is, Theorem \ref{thm:addendum} covers all of $\acc\{x_i\}_i$ and hence all recurrent behaviors. 

\paragraph{Oscillation compensation.} 
While the high-frequency oscillations (i.e., bouncing) will, in many cases, be considerable, they almost cancel out. This is what we refer to as oscillation compensation. The intuitive picture the reader should have in mind is a statement that the oscillations cancel out locally, as in \eqref{eq:intuitivecompensation}. %
Yet, because of small technical minutia, we do not have exactly  \eqref{eq:intuitivecompensation} and obtain instead very good approximations.
 Let us provide some explanations. 

Letting, in item \ref{ma:oscillationcompcont} of Theorem \ref{thm:main},
$\psi=\psi_{\delta,\eta}\colon\R^n\to[0,1]$ be a continuous cutoff function equal to 1 on a ball $B_\eta(x)$ of radius $\eta>0$ around a point $x\in\essacc\{x_i\}_i$ and vanishing outside the ball $B_\delta(x)$ for $\delta>\eta$, then we get, for appropriate subsequences $\{N_j\}_j\subset\N$,
\[\lim_{\delta\searrow0}\lim_{\eta\nearrow \delta}\lim_{j\to +\infty}\frac{\displaystyle\sum_{i=0}^{N_j}\varepsilon_iv_i\psi_{\delta,\eta}(x_i)}{\displaystyle\sum_{i=0}^{N_j}\varepsilon_i\psi_{\delta,\eta}(x_i)}=0,\]
which is indeed a very good approximation of \eqref{eq:intuitivecompensation}.

Similarly, setting, in item \ref{ad:oscillationcomp} of Theorem \ref{thm:addendum}, $U=B_{\eta}(x)$ and $V=B_{\delta}(x)$ the balls centered at $x $ with radius $0<\eta<\delta$, we obtain this  local version of the oscillation cancelation phenomenon: in the setting of Theorem \ref{thm:addendum} if $x\in\acc\{x_i\}_i$ and if $A_{\eta,\delta}\subset\N$ is the union of maximal intervals $I\subset \N$ such that $\{x_i\}_{i\in I}\in B_\delta(x)$ and $\{x_i\}_{i\in I}\cap B_\eta(x)\neq\emptyset$, then
\[\lim_{\delta\searrow 0}\lim_{\eta\nearrow\delta}\lim_{N\to+\infty} \frac{\displaystyle\sum_{\substack{0\leq i\leq N\\x_i\in A_{\eta,\delta}}}\varepsilon_iv_i}{\displaystyle\sum_{\substack{0\leq i\leq N\\x_i\in A_{\eta,\delta}}}\varepsilon_i}=0.\]
Note that as we take the limit $\eta\nearrow\delta$, we cover almost all $x_i$ in the ball $B_\delta(x)$, so we again get a statement very close to \eqref{eq:intuitivecompensation}. 

\paragraph{Convergence.}
While Theorem \ref{thm:addendum} tells us that $f(x_i)$ converges, we conjecture that this is no longer true in the context of Theorem \ref{thm:main}, which is a matter for future research. Similarly, in the setting of path differentiable functions, the question of determining whether all limit points of bounded sequences are critical remains open.

In all cases, including the Whitney stratifiable case, the sequence $\{x_i\}_i$ may not converge. A well-known example of such a situation was provided for the case of smooth $f$  by Palis--de Melo \cite{palisdemelo}.

However, our results show that the drift that causes the divergence of $\{x_i\}_i$ is very slow in comparison with the local oscillations. This slowness can be immediately appreciated in the statement of item \ref{ma:longloops} of Theorem \ref{thm:main} and items \ref{ad:longseparation} and \ref{ad:longintervals} of Theorem \ref{thm:addendum}. In substance, these results express that even if the sequence diverges, it takes longer and longer to connect disjoint neighborhoods of different limit points.

\section{A closed measure theoretical approach}
\label{sec:prelims}

Given an open subset of $\R^n$, denote by $C^0(U)$ the set of continuous functions while $C^p(U)$ is the set of $p\in [1,\infty]$ continuously differentiable functions. The set $\lipschitz (U)$ denotes the space of Lipschitz curves $\gamma\colon\R\to U$. When $U$ is bounded it is endowed with the supremum norm $\|\gamma\|_\infty=\sup_{t\in\R}\|\gamma(t)\|$. 
\subsection{A compendium on closed measures}
\label{sec:closedmeas}

\paragraph{General results.}

Given a measure $\xi$ on some set $X\neq \emptyset$ and a measurable map $g\colon X\to Y$, where $Y\neq\emptyset$ is another set, the \emph{pushfoward $g_*\xi$} is defined to be the measure on $Y$ such that, for $A\subset Y$ measurable, $g_*\xi(A)=\xi(g^{-1}(A))$.

Recall that the \emph{support $\supp\mu$} of a positive Radon measure $\mu$ on $\R^m$, $m\geq 0$,  is the set of points $x\in \R^m$ such that $\mu(U)>0$ for every neighborhood $U$ of $x$. It is a closed set.

The origin of the concept of closed measures (sometimes also called {\em holonomic measures} or {\em Young measures}) can be traced back to the work of L.C. Young \cite{young1969,young1937generalized} in the context of the calculus of variations. It has developed in parallel to the closely related normal currents \cite{federer,federerbook} and varifolds \cite{almgrenvarifolds,allard}, and has found applications in several areas of mathematics, especially Lagrangian and Hamiltonian dynamics \cite{matheractionminimizing91,manhe,contrerasiturriagabook,sorrentino2015action}, and also optimal transport \cite{patrick,bernardbuffoni}. 

The definition of closed measures is inspired from the following observations. Given a curve $\gamma\colon[a,b]\to\R^n$, its position-velocity information can be encoded by a measure $\mu_\gamma$ on $\R^n\times \R^n$ that is the pushforward of the Lebesgue measure on the interval $[a,b]$ into $\R^n\times\R^n$ through the mapping $t\mapsto (\gamma(t),\gamma'(t))$, that is,
\[\mu_\gamma=(\gamma,\gamma')_*\lebesgue_{[a,b]}.\]
In other words, if $\phi\colon\R^n\times\R^n\to\R$ is a measurable function, then the integral with respect to $\mu_\gamma$ is given by
\[\int_{\R^n\times\R^n}\phi(x,v)\,d\mu_\gamma(x,v)=\int_a^b\phi(\gamma(t),\gamma'(t))\,dt.\]
With this definition of $\mu_\gamma$ it follows that $\gamma$ is closed, that is, $\gamma(a)=\gamma(b)$ if, and only if, for all smooth $f\colon \R^n\to\R$, we have
\begin{multline*}\int_{\R^n\times \R^n}\nabla f(x)\cdot v\,d\mu_\gamma(x,v)=\int_a^b\nabla f(\gamma(t))\cdot\gamma'(t)\,dt %
=
\int_a^b(f\circ\gamma)'(t)dt=f\circ\gamma(b)-f\circ\gamma(a)=0.
\end{multline*}
In other words, the integral of $\nabla f(x)\cdot v$ with respect to $\mu_\gamma$ is exactly the circulation of the gradient vector field $\nabla f$ along the closed curve $\gamma$, and so it vanishes exactly when $\gamma$ is closed. This generalizes into:

\begin{defn}[Closed measure]\label{def:closedmeasure}
 A compactly-supported, positive, Radon measure $\mu$ on $\R^n\times\R^n$ is \emph{closed} if, for all functions $f\in C^\infty(\R^n)$, 
 \[\int_{\R^n\times\R^n}\nabla f(x)\cdot v\,d\mu(x,v)=0.\]
\end{defn}

Let $\pi\colon\R^n\times\R^n\to\R^n$ be the projection $\pi(x,v)=x$. To a measure $\mu$ in $\R^n\times\R^n$ we can associate its \emph{projected measure $\pi_*\mu$}. As an immediate consequence we have that $\supp\pi_*\mu=\pi(\supp\mu)\subseteq\R^n$. 

The disintegration theorem \cite{dellacherie1978probabilities} implies that there are probability measures $\mu_x$, $x\in\R^n$, on $\R^n$ such that
\begin{equation}\label{eq:desintegration}
 \mu=\int_{\R^n}\mu_x\,d(\pi_*\mu)(x). 
\end{equation}
We shall refer to the couple $(\pi_*\mu,\pi_x)$ as to the {\em desintegration} of $\mu$.
Thus if $\phi\colon\R^n\times\R^n\to\R$ is measurable, we have
\[\int_{\R^n\times\R^n}\phi\,d\mu=\int_{\R^n}\left[\int_{\R^n}\phi(x,v)\,d\mu_x(v)\right]d(\pi_*\mu)(x).\]
\begin{defn}[Centroid field]\label{def:centroid}
 Let $\mu$ be a positive, compactly-supported, Radon measure on $\R^n\times\R^n$. The \emph{centroid field $\bar v_x$} of $\mu$ is, for $x\in\R^n$ and with the decomposition \eqref{eq:desintegration},
 \[\bar v_x=\int_{\R^n} v\,d\mu_x(v).\]
\end{defn}

 The centroid field gives the average velocity, that is, the average of the velocities encoded by the measure at each point. As a consequence of the disintegration theorem \cite{dellacherie1978probabilities}, $x\mapsto\bar v_x$ is measurable, and for every measurable $\phi \colon \R^n\times\R^n \mapsto \R$ linear in the second variable, we have
\begin{align}\label{eq:linearityvbar}
    \int_{\R^n\times\R^n}\phi(x,v) \,d\mu(x,v)=\int_{\R^n} \phi(x,\bar v_x)\, d(\pi_*\mu)(x).
\end{align}
It plays a significant role in our work.  
For later use, we record the following facts that follow from the definition of the centroid field, the convexity of $\partial^cf(x)$, and the fact that $\mu_x$ is a probability:

\begin{lem}[Quasi-stationary bundle measures]\label{lem:centroidclosed}
 If a positive Radon measure $\mu$ has a centroid field $\bar v_x$ that vanishes $\pi_*\mu$-almost everywhere, then $\mu$ is closed. 
\end{lem}
\begin{proof}
 Indeed, if $\bar v_x=0$ for $\pi_*\mu$-almost every $x$, and if $f\in C^\infty(\R^n)$, we have
 \begin{align*}
  \int_{\R^n\times\R^n} \nabla f(x)\cdot v\,d\mu(x,v)
  &=\int_{\R^n} \int_{\R^n} \nabla f(x)\cdot v\,d\mu_x(v)\,d(\pi_*\mu)(x)\\
  &=\int_{\R^n}  \nabla f(x)\cdot\int_{\R^n} v\,d\mu_x(v)\,d(\pi_*\mu)(x)\\
  &=\int_{\R^n}  \nabla f(x)\cdot\bar v_x\,d(\pi_*\mu)(x)=0,
 \end{align*}
 so $\mu$ is closed.
\end{proof}

Recall that the weak* topology in the space of Radon measures on an open set $U$ is the one induced by the family of seminorms
\[|\mu|_f=\left|\int_U f\,d\mu\right|,\quad f\in C^0(U).\]
Thus a sequence $\{\mu_i\}_i$ of measure converges in this topology to a measure $\mu$ if, and only if, for all $f\in C^0(U)$,
\[\int_U f\,d\mu_i\to \int_U f\,d\mu.\]
The following result can be regarded as a consequence of the forthcoming Theorem \ref{thm:superposition}. It can also be seen as a special case of the results of \cite{federer} that are very well described in \cite[Theorem 1.3.4.6]{giaquintamodica}. Specifically it is shown in \cite[Theorem 1.3.4.6]{giaquintamodica} that it is possible to approximate, in a weak* sense, objects (namely, currents) intimately related to closed measures, by simpler objects (namely, closed polyhedral chains), which in our case correspond to combinations of finitely-many piecewise-smooth, closed curves.

\begin{prop}[Weak* density of closed curves]\label{prop:closedapproximation}
 Consider the set of measures 
 of the form $\beta\mu_\gamma$ for some $\beta>0$ and a measure $\mu_\gamma=(\gamma,\gamma')_*\lebesgue_{[a,b]}$ induced by some closed, smooth curve $\gamma\colon[a,b]\to\R^n$, $\gamma(a)=\gamma(b)$, defined on an interval $[a,b]\subset\R$ (which is not fixed). In the weak* topology, this set is dense in the set of closed measures.
\end{prop}

Since the space of measures is sequential, this proposition means that for any closed measure $\mu$, we can find a sequence of closed curves $\gamma_1,\gamma_2,\dots$ that approximate $\mu$ in the sense that $\mu_{\gamma_i}\to\mu$ in the weak* topology. 

The following result, known as the \emph{Young superposition principle} \cite{young1969,patrick} or as the \emph{Smirnov solenoidal representation} \cite{smirnov1993decomposition,bangert1999minimal}, is a strong refinement of the assertion of Proposition \ref{prop:closedapproximation}; see also \cite[Example 6]{rodolfo}. What this result tells us is basically that, not only can closed measures be approximated by measures induced by curves, but actually the centroidal measure 
\[\int \delta_{(x,\bar v_x)}\,d (\pi_*\mu)(x), \]
which captures much of the properties of $\mu$, can be decomposed into a combination of measures induced by Lipschitz curves. This decomposition is very useful theoretically, as there are no limits involved. For completeness, the following is proved in Section \ref{sec:pfsuperposition}.

\begin{thm}[Young superposition principle/Smirnov solenoidal representation]\label{thm:superposition}
 Let $U$ be a non\-emp\-ty bounded open subset of $\R^n$ and set $\lipschitz(U)=\lipschitz$. 
 For $t\in\R$, let $\tau_t\colon\lipschitz\to\lipschitz$ be the time-translation $\tau_t(\gamma)(s)=\gamma(s+t)$.
 For every closed probability measure $\mu$ supported in $U$ with centroid field $\bar v_x$, there is a Borel probability measure $\nu$ on the space $\lipschitz$ that is invariant under $\tau_t$ for all $t\in \R$ and such that
 \begin{equation}\label{eq:superposition}
  \int_{\R^n} \phi(x,\bar v_x)\,d(\pi_*\mu)(x)=\int_\lipschitz \phi(\gamma(0),\gamma'(0))\,d\nu(\gamma) 
 \end{equation}
 for any measurable $\phi\colon\R^n\times\R^n\to\R$.
\end{thm}
\noindent 
Curves lying in $\supp \nu$ have an appealing property:
\begin{coro}[Centroid representation]\label{cor:superpositionvx}
With the notation of the previous theorem, we have for $\nu$ almost all $\gamma$ in $\lipschitz$:
$$ \gamma'(t)=\bar v_{\gamma(t)}$$
for almost all $t$.
\end{coro}
\begin{proof}  Take indeed $\phi\geq 0$ vanishing only on the measureable set consisting of points of the form $(x,\bar v_{x})$, $x\in\R^n$. Then both sides of \eqref{eq:superposition} must vanish, which means that for $\nu$-almost all $\gamma$, the point $(\gamma(0),\gamma'(0))$ must be of the form $(x,\bar v_{x})$. The conclusion follows from the $\tau_t$-invariance of the measure $\nu$.
\end{proof}

As an example, take the case in which $\mu$ is the closed measure \[\mu=\frac1{2\pi}(\beta,\beta')_*\lebesgue_{[0,2\pi)}\] 
on $\R^2\times\R^2$ for
\[\beta(t)=(\cos t,\sin t).\]
In this simple example, the centroid coincides with the derivative, $\bar v_{\beta(t)}=\beta'(t)$.
Each time-translate $\tau_t(\beta)$ is still a parameterization of the circle, and the probability measure $\nu$ we obtain in Theorem \ref{thm:superposition} is 
\[\nu=\frac1{2\pi}\int_0^{2\pi} \delta_{\tau_t(\beta)}dt,\]
where $\delta_\gamma$ is the Dirac delta function whose mass is concentrated at the curve $\gamma$ in the space $\lipschitz$.

The measure $\nu$ in Theorem \ref{thm:superposition} can be understood as a decomposition of the closed measure $\mu$ into a convex superposition of measures induced by Lipschitz curves. Although at first sight each $\gamma$ on the right-hand side of \eqref{eq:superposition} only participates at $t=0$, the $\tau_t$-invariance of $\nu$ means that in fact the entire curve $\gamma$ is involved in the integral through its time translates $\tau_t\gamma$. 
Observe that another consequence of the $\tau_t$-invariance is that the integral in the right-hand side of \eqref{eq:superposition} satisfies, for all $t\in\R$,
\begin{align}
 \notag\int_{\lipschitz}\phi(\gamma(0),\gamma'(0))\,d\nu(\gamma)& = \int_{\lipschitz}\phi(\gamma(t),\gamma'(t))\,d\nu(\gamma)\\
 \label{eq:timetranslationinv}& =  \frac{1}{|I|}\int_I\int_{\lipschitz}\phi(\gamma(t),\gamma'(t))\,d\nu(\gamma)\, dt\\ \notag
 & = \frac{1}{|I|}\int_{\lipschitz}\int_I\phi(\gamma(t),\gamma'(t))\,dt\,d\nu(\gamma). \notag
\end{align}
where $I$ is any nontrivial interval. Thus  \eqref{eq:superposition} has the more explicit lamination or superposition form:
\begin{equation}\label{eq:superposition2}
  \int_{\R^n} \phi(x,\bar v_x)\,d(\pi_*\mu)(x)=\frac{1}{|I|}\int_{\lipschitz}\int_I\phi(\gamma(t),\gamma'(t))\,dt\,d\nu(\gamma) 
 \end{equation}
 for any interval $I$ with nonempty interior.
 \smallskip

Although the left-hand side of \eqref{eq:superposition} does not involve the full measure $\mu$, it will turn out to be similar enough: if the integrand $\phi\colon\R^n\times\R^n\to\R$ is linear in the second variable $v$, we still have \eqref{eq:linearityvbar} %
and this will be enough for the applications we have in mind.

We remark that the measure $\nu$ in Theorem \ref{thm:superposition} is not unique in general. For example, if $\gamma$ is a closed curve intersecting itself once so as to form the figure 8, then the measure $\nu$ decomposing $\mu=\mu_\gamma$ could be taken to be supported on all the $\tau_t$-translates of $\gamma$ itself, or it could be taken to be supported on the curves traversing each of the loops of the 8.

\paragraph{Circulation for a subdifferential field.}

We provide here some results related to subdifferentials, and that will be useful to the study of the vanishing step subgradient method.
\begin{lem}\label{lem:centroidclarke} Let $f$ be a locally Lipschitz continuous function and $\mu$ a closed measure with desintegration $(\pi_*\mu,\mu_x)$ and centroid field  $\bar v_x$. 
 If for some $a\in\R$ and some $x\in\R^n$ we have $a\supp\mu_x\subset \partial^cf(x)$, then $a\bar v_x\in\partial^cf(x)$. 
\end{lem}
\begin{proof}
 Assume $a\supp\mu_x\subset \partial^cf(x)$.
 Let $g(v)=\dist(av,\partial^cf(x))$, so that $g(v)=0$ for all $v\in\supp\mu_x$. Since $\partial^cf(x)$ is a convex set, $g$ is a convex function. Then by Jensen's inequality we have
 \[g(\bar v_x)=g\left(\int_{\R^n} v\,d\mu_x(v)\right)\leq \int_{\R^n} g(v)\,d\mu_x(v)=0.\qedhere\]
\end{proof}

\begin{prop}[Circulation of subdifferential for path differentiable functions]\label{prop:welldefined}
 If $f\colon\R^n\to\R$ is a path differentiable function and $\mu$ is a closed probability measure, then for each open set $U\subset\R^n$ and each measurable function $\sigma\colon U\to\R^n$ with $\sigma(x)\in\partial^cf(x)$ for $x\in U$, the integral
 \[\int_{U\times\R^n}\sigma(x)\cdot v\,d\mu(x,v)\]
 is well defined, and its value is independent of the choice of $\sigma$. We define the symbol
 \[\int_{U\times \R^n}\partial^cf(x)\cdot v\,d\mu(x,v)\]
 to be equal to this value. If $\pi(\supp\mu)\subset U$, 
  \[\int_{U\times\R^n}\partial^cf(x)\cdot v\,d\mu(x,v)=0.\]
\end{prop}
\begin{proof}
 Let $\sigma_1,\sigma_2\colon U\times\R^n\to\R$ be two measurable functions such that $\sigma_i(x)\in\partial^cf(x)$ for each $x\in U$. From Theorem \ref{thm:superposition} we get a $\tau_t$-invariant, Borel probability measure $\nu$ on the space $\lipschitz$ of Lipschitz curves. Then
 \begin{align*}
  \int_{U\times\R^n}&\sigma_1(x)\cdot v\,d\mu(x,v)-\int_{U\times\R^n}\sigma_2(x)\cdot v\,d\mu(x,v)\\
  &=\int_{\R^n\times\R^n}\chi_U(x)(\sigma_1(x)-\sigma_2(x))\cdot v\,d\mu(x,v)\\
  &=\int_{\lipschitz}\chi_U(\gamma(0))(\sigma_1(\gamma(0))-\sigma_2(\gamma(0)))\cdot \gamma'(0)\,d\nu(\gamma).
 \end{align*}
 Since $f$ is path differentiable, for each $\gamma\in\lipschitz$ and for almost every $t\in\R$ with $\gamma(t)\in U$,
 \[\sigma_1(\gamma(t))\cdot \gamma'(t)=\sigma_2(\gamma(t))\cdot \gamma'(t).\]
 From the $\tau_t$-invariance of $\nu$ it follows then that the integrand above vanishes $\nu$-almost everywhere.
 
 Let us now analyze the case in which $\pi(\supp\mu)\subset U$.
  Let $\psi\colon\R^n\to\R$ be a mollifier, that is, a compactly-supported, nonnegative, rotationally-invariant, $C^\infty$ function such that $\int_{\R^n}\psi=1$, and let $\psi_r(x)=r^{-n}\psi(x/r)$ for $r>0$, so that $\psi_r$ tends to the Dirac delta at 0 as $r\to0$. Denote by $\psi_r*f$ the convolution of $\psi_r$ and $f$. 
 Observe that if $\beta\in \lipschitz$ and $a<b$, then 
 \begin{multline*}
  \int_a^b(f\circ \beta)'(t)\,dt=f\circ\beta(b)-f\circ\beta(a)
  \\=\lim_{r\searrow 0}\left[(\psi_r*f)\circ\beta(b)-(\psi_r*f)\circ\beta(a)\right]
  =\lim_{r\searrow 0}\int_a^b((\psi_r*f)\circ\beta)'(t)\,dt.
 \end{multline*}
 This justifies the following calculation:
 \begin{align*}
  \int_{T\R^n}\partial^cf\,d\mu&=\int_{\lipschitz}(f\circ\beta)'(0)\,d\nu(\beta)=\\
  &=\lim_{r\searrow0}\int_{\lipschitz}((\psi_r*f)\circ\beta)'(0)\,d\nu(\beta) \\
  &=\lim_{r\searrow0}\int_{\lipschitz}\nabla(\psi_r*f)(\beta(0))\cdot\beta'(0)\,d\nu(\beta) \\
  &=\lim_{r\searrow0}\int_{T\R^n}\nabla(\psi_r*f)(x)\cdot v\,d\mu(x,v),
 \end{align*}
 which vanishes because $\mu$ is closed and $\psi_r*f$ is $C^\infty$.
\end{proof}

\subsection{Interpolant curves of subgradient sequences and their limit measures}
\label{sec:interpolatingcurve}

In this section we fix $f\colon\R^n\to\R$ to be a locally Lipschitz function, and let $\{x_i\}_i$ be a bounded sequence generated by the subgradient method.

\begin{defn}[Subgradient sequence interpolants]\label{def:interpolatingcurve}
 Given a sequence $\{x_i\}_{i\in\N}$ generated by the subgradient method, with the same notations as in Definition \ref{def:subgradient},
 its \emph{interpolating curve} is the curve $\gamma\colon\R_{\geq 0}\to\R^n$ with $\gamma(t_i)=x_i$ for $t_i=\sum_{j=0}^{i}\varepsilon_i$ and $\gamma'(t)=v_i$ for $t_i<t<t_{i+1}$. This curve corresponds to a continuous-time piecewise-affine interpolation of the sequence.%
\end{defn}

\medskip

For a bounded set $B\subset\R_{\geq0}$, we define a measure on $\R^n\times\R^n$ by
\[\meas{\gamma}{B}=\frac{1}{|B|}(\gamma,\gamma')_*\lebesgue_{B},\]
where $|B|=\int_B 1\,dt$ is the length of $B$, and $\lebesgue_B$ is the Lebesgue measure on $B$. 
If $\phi\colon\R^n\times\R^n\to\R$ is measurable, then
\[\int_{\R^n\times\R^n}\phi\,d\meas{\gamma}{B}=\frac1{|B|}\int_B\phi(\gamma(t),\gamma'(t))\,dt.\]

\begin{lem}[Limiting closed measures associated to  subgradient sequences]\label{lem:longintervals}
 Let $\gamma$ be the interpolating curve (as in Definition \ref{def:interpolatingcurve}) and $A=\{I_i\}_{i\in\N}$ be a collection of intervals $I_i\subset \R$, with disjoint interior, such that $|I_i|\to+\infty$ as $i\to+\infty$.
 Set $B_N=\cup_{i=0}^NI_i$.
 Then the set of weak* limit points of the sequence $\{\meas{\gamma}{B_N}\}_N$ is nonempty, and its elements are closed probability measures.

\end{lem}
\begin{proof}
 Let $\phi\in C^0(\R^n\times\R^n)$. For $i\in \N$, write $I_i=[t_1^i,t_2^i]$ and $d_i=\|\gamma(t_1^i)-\gamma(t_2^i)\|$, and let $\alpha_i\colon[0,d_i]\to\R^n$ be the segment joining $\gamma(t_2^i)$ to $\gamma(t_1^i)$ with unit speed. Also, let 
 \[\nu_i=(\alpha_i,\alpha'_i)_*\lebesgue_{[0,d_i]}\] 
 be the measure on $\R^n\times\R^n$ encoding $\alpha_i$. 
 Let $K\subset \R^n\times\R^n$ be a convex, compact set that contains the image of $(\gamma,\gamma')$ and $(\alpha_i,\alpha_i')$ for all $i$, so that $d_i\leq \diam K$. Estimate
 \begin{align*}
  \left|\frac{\sum_{i=0}^N\int_{\R^n\times\R^n} \phi\,d\nu_i}{|B_N|}\right|&=\left|\frac{\sum_{i=0}^N\int_{0}^{d_i} \phi(\alpha_i(t),\alpha_i'(t))\,dt}{\sum_{i=0}^N|I_i|}\right|\\
  &\leq \frac{N(\diam K)\sup_{(x,v)\in K}|\phi(x,v)|}{\sum_{i=0}^N|I_i|}\to 0
 \end{align*}
 since $|I_i|\to+\infty$. Thus the measures in the accumulation sets of the sequences $\{\meas{\gamma}{B_N}\}_N$ and 
 \begin{equation}\label{eq:measseq}\left\{\meas{\gamma}{B_N}+\frac{\sum_{i=0}^N\nu_i}{|B_N|}\right\}_N\end{equation}
 coincide. The measures in the latter sequence are all closed since, for all $\varphi\in C^\infty(\R^n)$, we have, by the fundamental theorem of calculus,
 \begin{align*}
  \int_{t^i_1}^{t^i_2} \nabla\varphi(\gamma(t))\cdot&\gamma'(t)\,dt+\int_0^{d_i}\nabla\varphi(\alpha(t))\cdot\alpha'(t)\,dt\\
  &=\int_{t_1^i}^{t_2^i} (\varphi\circ\gamma)'(t)\, dt + \int_0^{d_i} (\varphi\circ\alpha)'(t)\,dt
  \\
  &=[\varphi(\gamma(t^i_2))-\varphi(\gamma(t^i_1))]+[\varphi(\alpha(d_i))-\varphi(\alpha_i(0))]
  \\
  &=[\varphi(\gamma(t^i_2))-\varphi(\gamma(t^i_1))]+[\varphi(\gamma(t^i_1))-\varphi(\gamma(t^i_2))]=0,
 \end{align*}
 and the measures in the sequence \eqref{eq:measseq} are sums of multiples of these.
 
 By Prokhorov's theorem \cite{prokhorov}, the set of probability measures on $K$ is compact, so the set of limit points is nonempty.
 The set of closed measures is itself closed, as it is defined by a {weak* closed condition.
 Thus the limit points must also be closed measures.
}\end{proof}

\begin{lem}[Limit points and limiting measure supports]\label{lem:essaccsupp}
 Let $\gamma$ be the interpolating curve as in Definition \ref{def:interpolatingcurve}.
 Consider the set $\acc\{\meas{\gamma}{[0,N]}\}_N$ of limit points of the sequence $\{\meas{\gamma}{[0,N]}\}_N$ in the weak* topology. We have
 \[\overline{\bigcup_{\mu\in \acc\{\meas{\gamma}{[0,N]}\}_N}\pi(\supp \mu)}=\essacc\{x_i\}_i.\]
\end{lem}
\begin{proof}
 Let $B\subset\R^n$ be a closed ball containing a neighborhood of the sequence $\{x_i\}_i$.
 Let $\psi\colon\R^n\to\R_{\geq0}$ be a continuous function with $\supp\psi\subseteq B$. Since $\psi$ is uniformly continuous on $B$, given $\varepsilon>0$, there is $n_0>0$ such that $i>n_0$, $x,y\in B$, and $\|x-y\|\leq \varepsilon_i\lip(f)$ imply $|\psi(x)-\psi(y)|\leq \varepsilon$. We hence have $|\psi(x_i)-\psi(\gamma(t))|\leq \varepsilon$ for $t_i\leq t\leq t_{i+1}$ and $i>n_0$. Thus, for $S>n_0$,
 \[\left|\sum_{i=n_0}^S\varepsilon_i\psi(x_i)-\int_{t_{n_0}}^{t_S}\psi(\gamma(t))\,dt\right|\leq \varepsilon(t_S-t_{n_0}).\]
 
 Assume $x\in\essacc\{x_i\}_i\subset B$. Take a nonnegative continuous function $\psi$, as above, so that for all $R>0$ there is $S>R$ such that
 \[\frac{\sum_{1\leq i\leq S}\varepsilon_i\psi(x_i)}{\sum_{i=0}^S\varepsilon_i}>\delta.\]
 Then, for $\varepsilon=\delta/2$,  $n_0$ as above, and $S>R>n_0$,
 \begin{align*}
  \int_{\R^n\times\R^n}\psi\,d\meas{\gamma}{[0,S]}(x,v)
   &\geq\frac{1}{t_S} \int_{t_{n_0}}^{t_S}\psi(\gamma(t))\,dt \\
   &\geq 
  \frac{ \sum_{i=n_0}^S\varepsilon_i\psi(x_i)}{\sum_{i=0}^S\varepsilon_i}-\varepsilon \frac{t_S-t_{n_0}}{t_S}\\
   &>\delta-\varepsilon=\delta/2>0.
 \end{align*}
 It follows that there is some $\mu\in\acc\{\meas{\gamma}{[0,N]}\}_N$ with $\pi(\supp\mu)\cap\supp\psi\neq \emptyset$. 
 
 Observe that we can take the support of $\psi$ to be contained inside any neighborhood of $x$, so the argument above proves that there are measures in $\acc\{\meas{\gamma}{[0,N]}\}_N$ whose supports are arbitrarily close to $x$. 
 This proves the first inclusion.
 
  Conversely, assume that $x \in \overline{\bigcup_{\mu\in \acc\{\meas{\gamma}{[0,N]}\}_N}\pi(\supp \mu)}$. For a positive, continuous function $\psi$ with $x\in\pi(\supp\psi)$, there is $\mu \in \acc\{\meas{\gamma}{[0,N]}\}_N$ with $\int\psi\,d\mu > 0$. There is a subsequence of $\{\meas{\gamma}{[0,N]}\}_N$ converging to $\mu$, hence such that ${\sum_{1\leq i\leq S}\varepsilon_i\psi(x_i)}/{\sum_{i=0}^S\varepsilon_i}$ converges to a positive quantity, so that $x \in \essacc\{x_i\}_i$, and we obtain the opposite inclusion.
\end{proof}

The following corollary gives some connection between the discrete and the continuous subgradient systems. %

\begin{coro}[Limiting dynamics]\label{cor:discretecontinuous}
 Let $\{I_i\}_i$ be a sequence of disjoint, bounded intervals in $\R$ with
 \[\lim_{i\to+\infty}|I_i|=+\infty.\]
 Write $G_k=I_1\cup I_2\cup\dots \cup I_{k}$.
 Suppose that for some sequence $\{k_i\}_i\subset\N$, the limit \[\lim_{i\to+\infty}\meas{\gamma}{G_{k_i}}\] exists, so that, by Lemma \ref{lem:longintervals}, it is a closed probability measure $\mu$. 
 Let $\nu$ be a Borel probability measure on the space $\lipschitz$ of Lipschitz curves that is invariant under the time-translation $\tau_t$ and satisfies \eqref{eq:superposition}. Then $\nu$-almost every curve $\beta$ satisfies
 \[-\beta'(t)\in \partial^cf(\beta(t))\]
 for almost every $t\in\R$.
\end{coro}
\begin{proof}
 The existence of $\nu$ follows from Theorem \ref{thm:superposition}. By Corollary \ref{cor:superpositionvx}, we know that $\nu$-almost every curve $\beta\in\lipschitz$ satisfies,  $\beta'(t)=\bar v_{\beta(t)}$ for almost every $t$.   So we just need to prove that $\bar v_x\in-\partial^cf(x)$ for $\pi_*\mu$-almost every $x\in\R^n$. 
 
 Recall that $ \operatorname{graph}-\partial^cf=\{(x,v)\in\R^n\times\R^n:-v\in\partial^cf(x)\}.$ 
Let $t_i\leq t<t_{i+1}$, using the triangle inequality, 
the fact that $-\gamma'(t)$ is constant equal to $v_i$ in the interval $t\in[t_i,t_{i+1}]$ and belongs to $\partial^cf(\gamma(t_i))$, we have 
\begin{align*}
  \dist(&(\gamma(t),\gamma'(t)),\operatorname{graph}-\partial^cf)\\
 &\leq   \|(\gamma(t), \gamma'(t))-(\gamma(t_i), -v_i)\|+\dist((\gamma(t_i),-v_i),\operatorname{graph}-\partial^cf)\\
 &=  \|(\gamma(t),- v_i)-(\gamma(t_i),- v_i)\|+0\\
 &=   \|\gamma(t)-\gamma(t_i)\|\\
 & \leq   \lip(\gamma)\varepsilon_i\\
 &\leq \lip(f)\varepsilon_i.
\end{align*}
Now 
\begin{align*}
 \int_{\R^n\times\R^n}&\dist((x,v),\operatorname{graph}-\partial^cf)\,d\meas{\gamma}{G_{k_i}}(x,v)\\
 &=  \frac{1}{\sum_{j=1}^{k_i} |I_j|}\sum_{j=1}^{k_i} \int_{I_j}\dist\left((\gamma(t), \gamma'(t)),\operatorname{graph}-\partial^cf\right)dt\\
 & \leq   \frac{\lip(f)\sum_{j=1}^{k_i}  |I_j|\max_{t_\ell\in I_j} \varepsilon_\ell }{\sum_{j=1}^{k_i} |I_j|}.
\end{align*}
 This implies that 
 \[\lim_{i\to+\infty} \int_{\R^n\times\R^n}\dist((x,v),\operatorname{graph}-\partial^cf)\,d\meas{\gamma}{G_{k_i}}(x,v)=0\]
 by the Stolz-Ces\`aro theorem using the fact that, for $k$ large enough, $\sum_{j=1}^{k} |I_j| \geq ck$ for a positive constant $c$, and the fact that $\varepsilon_i$ converges to $0$ as $i \to + \infty$.
 This, in turn, implies that
 \[ \int_{\R^n\times\R^n}\dist((x,v),\operatorname{graph}-\partial^cf)\,d\mu(x,v)=0\]
 because the convergence of measures occurs in the weak* topology and the integrand is continuous. %
 Since $\operatorname{graph}-\partial^cf$ is a closed set, the support of $\mu$ must be contained in it. %
 From Lemma \ref{lem:centroidclarke} with $a=-1$, we know that $-\bar v_x\in \partial^cf(x)$, which is what we wanted to prove.
\end{proof}

\begin{thm}[Subgradient-like closed measures are trivial]\label{thm:vxvanishes}
 Assume that $f\colon\R^n\to\R$ is a path differentiable function.
 Let $\mu$ be a closed measure on $\R^n\times\R^n$, and assume that every $(x,v)\in\supp\mu$ satisfies $-v\in\partial^cf(x)$. Then the centroid field $\bar v_x$ of $\mu$ vanishes for $\pi_*\mu$-almost every $x$.
\end{thm}
\begin{proof}
 The condition on $\mu$ implies, by Lemma \ref{lem:centroidclarke} with $a=-1$, that $-\bar v_x\in\partial^cf(x)$.
 By Proposition \ref{prop:welldefined} we may choose $\sigma(x)=-\bar v_x$ to compute
 \begin{multline*}
  \int\partial^cf\,d\mu=\int \sigma(x)\cdot v\,d\mu(x,v)=\int \sigma(x)\cdot\left[\int v\,d\mu_x\right]d(\pi_*\mu)(x) \\
  =\int \sigma(x)\cdot \bar v_x\,d(\pi_*\mu)(x)=-\int \bar v_x\cdot \bar v_x\,d(\pi_*\mu)(x).
 \end{multline*}
 Proposition \ref{prop:welldefined} also implies that the left-hand side vanishes because $\mu$ is closed.
\end{proof}

\section{Proofs of main results}
\label{sec:proofs}

\subsection{Lemmas on the convergence of curve segments}
\label{sec:pflemmasegments}

\begin{lem}\label{lem:derivativeclarke}
 For each $i\in \N$, let $T_i>0$ and assume that $T_i\to T$ for some $T>0$. Let, for each $i\in\N$, $\gamma_i\colon [0,T_i]\to\R^n$ be a Lipschitz curve. Assume that the sequence $\{\gamma_i\}_i$ converges to some bounded, Lipschitz curve $\gamma\colon[0,T]\to \R$, $\gamma_i\to\gamma$, in the sense that $\sup_{t\in[0,\min(T_i,T)]}\|\gamma(t)-\gamma_i(t)\|\to 0$, and satisfies
 \begin{equation}\label{eq:closetoclarke}\lim_{i\to+\infty}\int_0^{T_i}\dist((\gamma_i(t),\gamma_i'(t)),\operatorname{graph}-\partial^cf)\,dt=0.
 \end{equation}
 Then $-\gamma'(t)\in\partial^cf(\gamma(t))$ for almost all $t\in[0,T]$.
\end{lem}
\begin{proof}
 We follow classical arguments; see for example \cite[Theorem 4.2]{BHS}. Let $0<T'<T$. For $i$ large enough, $T_i>T'$ because $T_i\to T$. In particular, we eventually have uniform convergence of $\gamma_i$ on $[0,T']$ to the restriction of $\gamma$ to $[0,T']$. 
 For each $i$, the derivative $\gamma_i'$ is an element of $L^\infty=L^\infty([0,T'];\R^d)$, and being uniformly bounded with compact domain, belong to $L^2=L^2([0,T'];\R^d)$ as well.
 Recall that, since $L^2$ is reflexive, the weak and weak* topologies coincide in $L^2$.
 So by the Banach--Alaoglu compactness theorem, by passing to a subsequence we may assume that $\gamma'_j$ converge weakly in $L^2$ and weak* in $L^\infty$ to some $u\in L^2 \cap L^\infty$.
 
 Since $\gamma_j$ converges to $\gamma$ uniformly, $\gamma_j\to\gamma$ also in $L^2$. Hence $\gamma'_j$ tends to $\gamma'$ in the sense of distributions on $[0,T']$; indeed, for all $C^\infty$ functions $g\colon[0,T']\to\R$ with compact support in $(0,T')$, we have 
 \begin{equation*}
  \int_0^{T'}\gamma'_j(t)g(t)\,dt=-\int_0^{T'}\gamma_j(t)g'(t)\,dt
  \to-\int_0^{T'}\gamma(t)g'(t)\,dt
  =\int_0^{T'}\gamma'(t)g(t)\,dt
 \end{equation*}
 since we have convergence in $L^2$. By uniqueness of the limit, $u=\gamma'$ almost everywhere on $[0,T']$.

It follows from Mazur's lemma \cite[p. 6]{ekelandtemam} that there is a function $N\colon\N\to\N$ and, for each $p\leq k\leq N(p)$, a number $a(p,k)\geq 0$ such that $\sum_{k=p}^{N(p)}a(p,k)=1$, and such that the convex combinations \begin{equation}\label{eq:mazur}
 \sum_{k=p}^{N(p)}a(p,k)\gamma'_k\to\gamma'
\end{equation}
 strongly in $L^2$ as $p\to +\infty$ (and also in the weak* sense in $L^\infty$). 
 Since the Clarke subdifferential $\partial^cf(x)$ is convex at each $x$, the function 
  \[g(x,v)=\dist(-v,\partial^cf(x))\]
 is convex in its second argument for fixed $x \in \R^n$.    Using the fact that the convergence \eqref{eq:mazur} happens pointwise almost everywhere, we have, by continuity of $g$ and by the fact that countable union of zero measure sets has zero measure, for almost all $t \in [0, T']$
 \begin{align*}
  g(\gamma(t),\gamma'(t)) &=g({\gamma(t),\lim_{p\to+\infty}\textstyle\sum_{k=p}^{N(p)}a(p,k)\gamma'_k(t)})\\
  &=\lim_{p\to+\infty}g({\textstyle\gamma(t),\sum_{k=p}^{N(p)}a(p,k)\gamma'_k(t)})\\
  &\leq \liminf_{p\to+\infty} \sum_{k=p}^{N(p)}a(p,k)g(\gamma(t),\gamma'_k(t)),
 \end{align*}
 where the last step follows from Jensen's inequality and convexity of $g$ in its second argument.
 Since $g$ is non negative, integrating on $[0,T']$, we have using Fatou's Lemma, 
 \begin{align*}
  0&\leq \int_0^{T'}g(\gamma(t),\gamma'(t))\,dt\\
  &\leq \liminf_{p\to+\infty}\int_0^{T'}\sum_{k=p}^{N(p)}a(p,k)g(\gamma(t),\gamma'_k(t))\,dt\\
  &\leq\liminf_{p\to+\infty}\int_0^{T'}\sum_{k=p}^{N(p)}a(p,k)[\dist((\gamma(t),\gamma'_k(t)),(\gamma_k(t),\gamma'_k(t)))\\
  &\hspace{6.5cm}+g(\gamma_k(t),\gamma'_k(t))]\,dt\\
  &=\liminf_{p\to+\infty}\int_0^{T'}\sum_{k=p}^{N(p)}a(p,k)[\dist(\gamma(t),\gamma_k(t))+g(\gamma_k(t),\gamma'_k(t))]\,dt\\
   \end{align*}
 where we have used the triangle inequality.
 Now, using a uniform bound on the integral, we have 
 \begin{align*}
  0&\leq \int_0^{T'}g(\gamma(t),\gamma'(t))\,dt\\
  &\leq\liminf_{p\to+\infty}\sum_{k=p}^{N(p)}a(p,k)\left(T'\sup_{t\in[0,T']}[\dist(\gamma(t),\gamma_k(t))]+ \int_0^{T'} g(\gamma_k(t),\gamma'_k(t))\right)\\
  &\leq \liminf_{p\to+\infty}\sup_{p\leq k\leq N(p)}\left(T'\sup_{t\in[0,T']}[\dist(\gamma(t),\gamma_k(t))]+ \int_0^{T'} g(\gamma_k(t),\gamma'_k(t))\right)\\
  &\leq \limsup_{k\to+\infty} \left(T'\sup_{t\in[0,T']}[\dist(\gamma(t),\gamma_k(t))]+ \int_0^{T'} g(\gamma_k(t),\gamma'_k(t))\right)=0,
 \end{align*}
 where we used the fact that $\sum_{k=p}^{N(p)}a(p,k)=1$, the fact
 that $\gamma_k\to\gamma$ uniformly and the hypothesis in \eqref{eq:closetoclarke}. Hence we have $-\gamma'(t)\in\partial^cf(\gamma(t))$ for almost all $t\in[0,T']$, and this proves the lemma since $T'$ was taken arbitrary in $(0,T)$.
\end{proof}

\begin{lem}\label{lem:generaltechnique}
 Let $\gamma$ be the interpolant curve of the bounded gradient sequence $\{x_i\}_i$, and let $\{I_j\}_j$ be a collection of pairwise-disjoint intervals of $\R_{\geq0}$ of  length $1/C \leq |I_j| \leq C$ for some $C>1$. 
 Then there is a subsequence $\{j_k\}_k\subset\N$ such that the restrictions $\gamma|_{I_{j_k}}$ converge uniformly to a Lipschitz curve $\bar\gamma\colon[a,b]\to\R$ that satisfies
 \[-\int_a^b\|\bar\gamma'(t)\|^2dt=f\circ\bar\gamma(b)-f\circ\bar\gamma(a).\]
\end{lem}
\begin{proof}
 By passing to a subsequence, we may assume that the lengths $|I_j|$ converge to a positive number. 
 By the Lipschitz version of the Arzel\`a--Ascoli theorem, we may pass to a subsequence such that $\gamma|_{_{I_{j_k}}}$ converges uniformly to a curve $\bar \gamma$ on an interval $[a,b]$ of length $\lim_{j\to+\infty}|I_j|>0$. Condition \ref{eq:closetoclarke} holds if we let $\gamma_i$ be the appropriate translate of $\gamma|_{I_i}$, so by Lemma \ref{lem:derivativeclarke}, $-\bar\gamma'(t)\in\partial^cf(\bar\gamma(t))$ for almost every $t\in[a,b]$. By the path differentiability of $f$, we have
 \begin{align*}
  -\int_a^b\|\bar\gamma'(t)\|^2dt&=\int_a^b\partial^cf(\bar\gamma(t))\cdot\bar\gamma'(t)dt\\
  &=\int_a^b(f\circ\bar\gamma)'(t)\,dt=f\circ\bar\gamma(b)-f\circ\bar\gamma(a). \qedhere
 \end{align*}
\end{proof}

\subsection{Proof of Theorem \ref{thm:main}}
\label{sec:proofmain}

\subsubsection{Item \ref{ma:longloops}}
\label{sec:proofmain0}
Let $\gamma$ be the interpolant curve of the sequence $\{x_i\}_i$, and consider the intervals $I_k=[t_{i_k},t_{i'_k}]$, so that the endpoints of the restriction $\gamma|_{_{I_k}}$ are precisely $\gamma(t_{i_k})=x_{i_k}$ and $\gamma(t_{i'
_k})=x_{i'_k}$. Aiming for a contradiction, assume that the numbers $\bar T_k=t_{i'_k}-t_{i_k}$ remain bounded. Apply Lemma \ref{lem:generaltechnique} to obtain a curve $\bar\gamma\colon[a,b]\to\R^n$ joining $\bar\gamma(a)=\lim_kx_{i_k}=x$ and $\bar\gamma(b)=\lim_kx_{i'_k}=y$. So we have that the arc length of $\bar\gamma$ must be positive because $x\neq y$, while $\bar\gamma$ also satisfies, as part of the conclusion of Lemma \ref{lem:generaltechnique},
\[0\geq -\int_a^b\|\bar\gamma'(t)\|^2dt=f\circ\bar\gamma(b)-f\circ\bar\gamma(a)=f(y)-f(x)\geq0.\]
Whence we get the contradiction we were aiming for.

\subsubsection{Item \ref{ma:oscillationcompcont}}
\label{sec:proofmain1}
Let $B\subset\R^n$ be a closed ball containing the sequence $\{x_i\}_i$. By convexity, $B$ contains also the image of the interpolating curve $\gamma$.

Fix $\varepsilon>0$. By uniform continuity of $\psi$ over $B$, there exists $n_0>0$ such that $i>n_0$, $x,y\in B$ and $|x-y|\leq \varepsilon_i\lip(f)$ imply $|\psi(x)-\psi(y)|\leq \varepsilon$. We hence have $|\psi(x_i)-\psi(\gamma(t))|\leq\varepsilon$ for $t_i\leq t\leq t_{i+1}$ and $i>n_0$. Thus 
\[\left|\sum_{i=n_0}^{N_j}\varepsilon_iv_i\psi(x_i)-\int_{t_{n_0}}^{t_{N_j}}\gamma'(t)\psi(\gamma(t))dt\right|\leq \varepsilon \lip(f)(t_{N_j}-t_{n_0}) \]
and
\begin{multline*}
 \frac{1}{\sum_{i=0}^{N_j}\varepsilon_i}
 \left|\sum_{i=0}^{N_j}\varepsilon_iv_i\psi(x_i)-\int_{0}^{t_{N_j}}\gamma'(t)\psi(\gamma(t))dt\right|\\
 \leq \frac{1}{\sum_{i=0}^{N_j}\varepsilon_i}\left[\left|\sum_{i=0}^{n_0-1} \varepsilon_i v_i\psi(x_i)-\int_{0}^{t_{n_0}}\gamma'(t)\psi(\gamma(t))dt\right| +\varepsilon \lip(f)(t_{N_j}-t_{n_0})\right].
\end{multline*}
Since $\sum_{i=0}^\infty\varepsilon_i=+\infty$ and $\varepsilon > 0$ was arbitrary, it follows that the latter becomes arbitrarily small as $N_j$ grows. 

Whence the quotient in the limit in the statement of item \ref{ma:oscillationcompcont} is very close, for large $j$, to 
\[\frac{\sum_{i=0}^{N_j}\varepsilon_i}{\sum_{i=0}^{N_j}\varepsilon_i\psi(x_i)}\int_{\R^n\times\R^n}v\psi(x)d\meas{\gamma}{[0,t_{N_j+1}]}(x,v).\]
We now prove that the above quantity converges to 0  as $j\to+\infty$.
 Taking a subsequence so that $\meas{\gamma}{[0,t_{N_j+1}]}$ converges to some probability measure $\mu$, the quotient on the left converges to 
 \[\left.1\middle/\int \psi(x)\,d\pi_*\mu(x),\right.\] 
 and our hypothesis on the subsequence $\{N_j\}_j$ thus guarantees that $\int \psi(x)\,d\pi_*\mu(x)>0$. %

Thus, it suffices to show that, for every limit point $\mu$ of the sequence $\{\meas{\gamma}{[0,t_{N+1}]}\}_N$ satisfying $$\displaystyle \int \psi(x)\,d\pi_*\mu(x)>0,$$ we have
\begin{equation}\label{eq:wantintegral}
\int_{\R^n\times\R^n} v\psi(x)\,d\mu(x,v)=\int_{\R^n} \bar v_x\psi(x)\,d(\pi_*\mu)(x)=0,
\end{equation}
where $\bar v_x$ is the centroid field of $\mu$. By Lemma \ref{lem:longintervals} we know that $\mu$ is closed so that Theorem \ref{thm:vxvanishes} applies which gives  $\bar v_x=0$ for $\pi_*\mu$-almost every $x$. This immediately implies \eqref{eq:wantintegral}.
\subsubsection{Item \ref{ma:criticality}}
\label{sec:proofmain2}

To prove item \ref{ma:criticality}, consider the interpolation curve constructed in Section \ref{sec:interpolatingcurve}. 
Consider a limit point $\mu$ of the sequence $\{\meas{\gamma}{[0,N]}\}_N$. By Lemma \ref{lem:longintervals}, $\mu$ is closed.
By Theorem \ref{thm:vxvanishes}, the centroid field $\bar v_x$ of $\mu$ vanishes for $\pi_*\mu$-almost every $x$, so from Lemma \ref{lem:centroidclarke} we know that $0=\bar v_x\in\partial^cf(x)$, and hence $x\in \crit f$ for a dense subset of $\pi(\supp\mu)$. Since this is true for all limit points $\mu$, by Lemma \ref{lem:essaccsupp} we know that it is true throughout $\essacc\{x_i\}_i$.

\subsection{Proof of Theorem \ref{thm:addendum}}
\label{sec:proofaddendum}

 \subsubsection{The function is constant on the accumulation set}
 \label{sec:accset}

\begin{lem}\label{lem:acc}
 Assume that the path differentiable function $f\colon\R^n\to\R$ is constant on each connected component of its critical set, and let $\{x_i\}_i$ be a bounded sequence produced by the subgradient method. Then $f$ is constant on the set $\operatorname{acc}\{x_i\}_i$ of limit points of $\{x_i\}_i$.
\end{lem}
\begin{proof}
 Assume instead that $f$ takes two values $J_1<J_2$ within $\operatorname{acc}\{x_i\}_i$. 
 
 Let $K$ be a compact set that contains the closure $\overline{\{x_i\}_i}$ in its interior. %
 Since $f$ is constant on the connected components of $\operatorname{crit}f$ and since $f$ is Lipschitz, the set $f(K\cap\operatorname{crit}f)$ has measure zero because, given $\varepsilon>0$, the connected components $C_i$ of $K\cap\operatorname{crit}f$ of positive measure $|C_i|>0$ ---of which there are only countably many--- can be covered with open sets \[f^{-1}((f(C_i)-\varepsilon/2^{i+1},f(C_i)+\varepsilon/2^{i+1}))\] 
 with image under $f$ of length $\varepsilon/2^i$; the rest of $K\cap\operatorname{crit} J$ has measure zero, so it is mapped to another set of measure zero. The set $f(K\cap\operatorname{crit}f)$ is also compact, so we conclude that it is not dense on any open interval of $\R$. 
 
 We may thus assume, without loss of generality, that the values $J_1$ and $J_2$ are such that there are no critical values of $f|_K$ between them.
 
 Pick $c_1,c_2\in\R$ such that
 \[J_1<c_1<c_2<J_2.\]
 Let 
 \[W_1=f^{-1}(-\infty,c_1)\quad\textrm{and}\quad W_2=f^{-1}(c_2,+\infty).\]
 Clearly $W_j\cap\operatorname{acc}\{x_i\}_i\neq\emptyset$ because the value $J_j$ is attained in $\operatorname{acc}\{x_i\}_i$, $j=1,2$. 
 
 Consider the curve $\gamma\colon\R_{\geq0}\to K\subset\R^n$ interpolating the sequence $\{x_i\}_i$. Let $A$ be the set of intervals
 \begin{equation*}
  A=\{[t_1,t_2]\subset\R:t_1<t_2,\gamma(t_1)\in\partial W_1,\gamma(t_2)\in \partial W_2,\\ 
  \textrm{$\gamma(t)\notin \overline{W_1\cup W_2}$ for $t\in (t_1,t_2)$}\}
 \end{equation*}
 Write $A=\{I_j\}_{j\in\N}$ for maximal, disjoint intervals $I_j$.
 Observe that if $I_j=[t_1^j,t_2^j]$, then we have, by the path-differentiability of $f$, that 
 \begin{equation}\label{eq:singleintervalint}\int_{t_1^j}^{t_2^j}\partial^c f(\gamma(t))\cdot\gamma'(t)\,dt=\int_{t_1^j}^{t_2^j}(f\circ\gamma)'(t)\,dt
 =f\circ\gamma(t_2^j)-f\circ\gamma(t_1^j)=c_2-c_1. 
 \end{equation}
 Let $\mu$ be a probability measure that is a limit point of the sequence $\{\meas{\gamma}{\cup_{i=0}^NI_i}\}_N$.

 Now, since $f$ is Lipschitz and $\overline W_1$ and $\overline W_2$ are compact, $|I_i|$ is bounded from below, let us say
 \[|I_i|>\alpha.\] 
 It is also bounded from above, because if not then we there is a subset $\{I_{i_j}\}_j$ of $A$ consisting of intervals  with length $|I_{i_j}|\to +\infty$, and we can apply Lemma \ref{lem:longintervals} and Theorem \ref{thm:vxvanishes} to get closed measures $\tilde\mu$ with 
 $\supp\pi_*\tilde\mu\subset\operatorname{crit}f$. Since the support of each such $\pi_*\tilde\mu$ is contained in $K\setminus (W_1\cup W_2)$, this would mean the existence of a critical value between $c_1$ and $c_2$, which contradicts our choice of $J_1$ and $J_2$. We conclude that the size of the intervals in $A$ is also bounded from above, say,
 \[|I_i|<\beta.\]
 By \eqref{eq:singleintervalint} we have
 \[\int \partial^cf\,d\meas{\gamma}{\cup_{i=0}^NI_i}=\frac{N(c_2-c_1)}{\sum_{i=0}^N|I_i|}\]
 and
 \[0<\frac{c_2-c_1}{\beta }\leq \frac{N(c_2-c_1)}{\sum_{i=0}^N|I_i|}\leq \frac{c_2-c_1}{\alpha}.\]
 Whence we also have at the limit
 \begin{equation}\label{eq:ftc}\int \partial^cf\,d\mu\geq \frac{c_2-c_1}{\beta}>0.\end{equation}
 
 Let $\bar v_x$ denote the centroid velocity vector field for $\mu$. By the construction of $\{x_i\}_i$, the vector $-\bar v_x$ is contained in the Clarke subdifferential of each point of $\supp\pi_*\mu$, and since the path-differentiability of $f$ allows us to choose any representative of this differential, we have, as in the proof of Theorem \ref{thm:vxvanishes},
 \[\int\partial^cf\,d\mu=-\int\bar v_x\cdot \bar v_x\,d\pi_*\mu\leq 0,\]
 which contradicts \eqref{eq:ftc}.
\end{proof}

\subsubsection{Proof of item \ref{ad:longseparation}}
\label{sec:pfaddendum5}

For $j\in\N$, let  $I_j=[t_{i_j},t_{i_{j+1}}]\subset\R$ be the interval closest to 0 with $t_j\leq t_{i_j}<t_{i_{j+1}}$, $\gamma(t_{i_j})\in B_\delta(x)$, and $\gamma(t_{i_{j+1}})\in B_\delta(y)$, so that $T_j=t_{i_{j+1}}-t_{i_j}$. Let $\gamma|_{I_j}$ be the restriction of the interpolant curve $\gamma$. Since the two balls $B_\delta(x)$ and $B_\delta(y)$ are at positive distance from each other, and since the velocity is bounded uniformly $\|\gamma'\|\leq \lip(f)$, we know that the numbers $T_j=|I_j|$ are uniformly bounded from below by a positive number.

Assume, looking for a contradiction, that there is a subsequence of $\{T_j\}_j$ that remains bounded from above. Apply Lemma \ref{lem:generaltechnique} to obtain a curve $\bar\gamma\colon[a,b]\to\R$ such that $\bar\gamma(a)\in B_{\delta}(x)\cap\acc\{x_i\}_i$ and $\bar\gamma(b)\in B_\delta(y)\cap\acc\{x_i\}_i$, while also satisfying
 \begin{equation}\label{eq:finalcalc}
 -\int_a^b \|\bar\gamma'(t)\|^2dt =f\circ\bar\gamma(b)-f\circ\bar\gamma(a)=0
\end{equation}
by Lemma \ref{lem:acc}. This contradicts the fact that the distance between the balls $B_\delta(x)$ and $B_\delta(y)$ ---and hence also the arc length of $\bar\gamma$--- is positive.

\subsubsection{Proof of item \ref{ad:longintervals}}
\label{sec:pfaddendum4}

 Aiming at a contradiction, we assume instead that there is some $x\in \overline U\cap\operatorname{acc}\{x_i\}_i$ and some subsequence $\{i_j\}_j$ such that $\dist(x,\gamma(I_{i_j}))\to 0$ and $|I_{i_j}|\leq C$ for some $C>0$ and all $j\in\N$.

 We may thus apply Lemma \ref{lem:generaltechnique} to get a curve $\bar\gamma:[a,b]\to\R^n$ whose endpoints $\bar\gamma(a)$ and $\bar\gamma(b)$ are contained in $\acc\{x_i\}_i\setminus V$, and $\bar\gamma$ passes through $x\in \overline U$, so it has positive arc length. However, it is also a conclusion of Lemma \ref{lem:generaltechnique}, together with Lemma \ref{lem:acc}, that $\bar\gamma$ satisfies \eqref{eq:finalcalc}, which makes it impossible for its arc length to postive, so we have arrived at the contradiction we were looking for.

\subsubsection{Proof of item \ref{ad:oscillationcomp}}
\label{sec:pfaddendum1}

Let $U$, $V$, and $A$ be as in the statement of item \ref{ad:oscillationcomp}. 
Let $B=\cup_{i\in A}[t_i,t_{i+1})$. 
The statement of item \ref{ad:oscillationcomp} is equivalent to the statement that 
\begin{equation}\label{eq:weneed}\lim_{N\to+\infty}\int v\,d\meas{\gamma}{B\cap[0,N]}=0.\end{equation}
It follows from item \ref{ad:longintervals} that the maximal intervals $I_i\subset \R$ comprising $B=\bigcup_iI_i$ satisfy $|I_i|\to +\infty$. Hence, from Lemma \ref{lem:longintervals} we know that any limit point $\mu$ of the sequence $\{\meas{\gamma}{B\cap[0,N]}\}_N$ is closed, and from Theorem \ref{thm:vxvanishes} we know that the centroid field of $\mu$ vanishes $\pi_*\mu$-almost everywhere, which implies \eqref{eq:weneed}.

\subsubsection{Proof of item \ref{ad:criticality}}
\label{sec:pfaddendum2}

Let $x\in\acc\{x_i\}_i$.
For any neighborhood $U$ of $x$, we can take a slightly larger neighborhood $V$ and repeat the construction described in the proof of item \ref{ad:oscillationcomp} (Section \ref{sec:pfaddendum1}) of a closed measure $\mu$ whose support intersects $U$, and whose centroid field vanishes $\pi_*\mu$-almost everywhere. By Lemma \ref{lem:centroidclarke} we know that the centroid field is contained in the Clarke subdifferential. In sum, we have that in every neighborhood $U$ of $x$, there is a point $y\in U$ with $0\in \partial^c f(y)$, which implies that $0\in \partial^cf(x)$ because the graph of $\partial^cf$ is closed in $\R^n\times\R^n$.

\subsubsection{Proof of item \ref{ad:convergencef}}
\label{sec:pfaddendum3}

Recall that  $\acc\{x_i\}_i$ is connected. %
We know from item \ref{ad:criticality} that $\acc\{x_i\}_i\subseteq\crit f$.
So it is contained in a single connected component of $\crit f$. Hence $f$ must be constant on $\acc\{x_i\}_i$, and $\{f(x_i)\}_i$ converges.

\subsection{Proof of Corollary \ref{cor:perp}}
\label{sec:pfcorperp}

Let $\alpha$ be as in the statement of the corollary. By item \ref{ma:criticality} of Theorem \ref{thm:main} or of Theorem \ref{thm:addendum}, we know that $0\in\partial^cf(\alpha(0))$. Since $-v_i\in \partial^cf(x_i)$ and because the graph of $\partial^cf$ is closed in $\R^n\times\R^n$, we have that $-v\in\partial^cf(\alpha(0))$. The path differentiability of $f$ means that the choice of element of $\partial^cf(\alpha(0))$ is immaterial when we compute $(f\circ\alpha)'(0)$. %
So we have
\[w\cdot\alpha'(0)=(f\circ\alpha)'(0)=0\cdot\alpha'(0)=0.\]

\subsection{Proof of Corollary \ref{cor:whitney}}
\label{sec:pfcorwhitney}
 
  Let $K$ be a compact set that contains $\{x_i\}_i$ in its interior. By the Morse--Sard theorem applied independently on each stratum of the stratification of $f$, $f(\crit f)$ is a compact set of measure zero. Thus, it must be a totally-separated subset of $\R$. It follows that $f$ is constant on each connected component of $\crit f$. In other words, we are in the setting of Theorem \ref{thm:addendum}. From item \ref{ad:criticality} of Theorem \ref{thm:addendum} we know that $x\in\acc\{x_i\}_i\subset\crit f$, and the additional condition we have on $x$ tells us that $\partial^c f(x)\neq\{0\}$, so $x$ must be contained in a stratum of smaller dimension. The last statement of the corollary follows from Corollary \ref{cor:perp}.

\appendix

\section{Proof of Theorem \ref{thm:superposition}}
\label{sec:pfsuperposition}

\noindent We follow the exposition of \cite{bangert1999minimal}.

We remark at the outset that if $\nu$ is a probability measure on $\lipschitz$ that is $\tau_t$-invariant for all $t\in\R$, then in view of \eqref{eq:timetranslationinv}, for $\varphi\in C^\infty(\R^n)$,
\begin{multline*}
    \int_{\lipschitz}\nabla \varphi(\gamma(0))\cdot\gamma'(0)\,d\nu(\gamma)=\int_0^1\int_{\lipschitz} \nabla \varphi(\gamma(t))\cdot\gamma'(t)\,d\nu(\gamma)\,dt\\
      =\int_{ \lipschitz} \int_0^1(\varphi\circ\gamma)'(t)\,dt\,d\nu(\gamma)
      =\int_{ \lipschitz} (\varphi\circ\gamma(1)-\varphi\circ\gamma(0))\,d\nu(\gamma)=0.
\end{multline*}
Thus the probability measure induced by $\nu$ on $\R^n\times\R^n$ by pushing forward through the map $\gamma\mapsto (\gamma(0),\gamma'(0))$ (as in the right-hand side of \eqref{eq:superposition}) is automatically closed.

\smallskip
\noindent \textbf{Smooth case.}
Assume first that there is a $C^\infty$ compactly-supported vector field $X\colon U\to\R^n$ such that $\mu$ is given by $\delta_{(x,X(x))}\otimes\rho(x)\lebesgue_{U}(x)$, with some smooth probability density $\rho\colon\R^n\to\R_{\geq0}$, 
\[\int_{\R^n\times\R^n}\phi\,d\mu=\int_{\R^n}
\phi(x,X(x))\,\rho(x)\,dx,\]
for measurable $\phi\colon\R^n\to\R$.
For $\mu$, the centroid field is $\bar v_x=X(x)$. Without loss of generality we may assume that $X$ vanishes in a neighborhood of the boundary $\partial U$.

Denote by $\Phi\colon U\times\R\to U$ the flow of $X$, so that, for all $t\in\R$ and writing $\Phi_t(x)=\Phi(x,t)$,
\[\Phi_0(x)=x\quad\textrm{and}\quad \frac{d}{dt}\Phi_t(x)=X(\Phi_t(x)).\]
Since $\supp X$ is compact, by the Picard--Lindel\"of theorem we know that $\Phi_t(x)$ is defined for all $t\in\R$ for all $x\in U$.
The measure $\mu$ is $\Phi_t$-invariant because, integrating by parts, we get that for all $\varphi\in C^\infty(\R^n)$,
\begin{equation*}
0=\int_{\R^n\times\R^n} \nabla \varphi(x)\cdot v\,d\mu(x,v)=\int_{\R^n}\nabla\varphi(x)\cdot X(x)\,\rho(x)\,dx=-\int\varphi(x)\divergence(\rho X)(x)\,dx,
\end{equation*}
so $\rho X$ is a divergence-free vector field; thus, the flow $\Phi_t$ of $X$ preserves $\rho$.

For $L>0$, let $\lipschitz_L\subset\lipschitz$ be the set of Lipschitz curves $\gamma$ with Lipschitz constant at most $L$. Then 
\[\Gamma=\{\gamma\colon\R\to U|\textrm{ $\gamma(t)=\Phi_t(x)$ for some $x\in U$ and all $t\in\R$}\}\subset\lipschitz_{\|X\|_\infty}\]
and $\Gamma$ is a Borel subset of $\lipschitz$ because it can be expressed as the countable intersection of unions of the closed balls around a dense subset of $\Gamma$. 

Let $\operatorname{ev}\colon\lipschitz\to \R^n$ be the evaluation at 0, namely, $\operatorname{ev}(\gamma)=\gamma(0)$.
Denote by $\operatorname{ev}\!|_\Gamma^{-1}\colon U \to \Gamma$ the inverse of the one-to-one map that results from restricting $\operatorname{ev}$ to $\Gamma$.  For $x\in U$, $\operatorname{ev}|_{\Gamma}^{-1}(x)$ is exactly the curve $\beta\in\Gamma$ given by $\beta(t)=\Phi_t(x)$ that satisfies, in particular, $\operatorname{ev}(\beta)=\beta(0)=x$ and $\beta'(0)=X(x)$.

Let $\nu$ be the pushforward 
\[\nu=(\operatorname{ev}\!|_\Gamma^{-1})_*(\rho(x)\lebesgue_{\R^n}(x)) \]
so that, for measurable $\phi\colon\R^n\times\R^n\to\R$,
\begin{equation*}
\int_{\lipschitz}\phi(\gamma(0),\gamma'(0))\,d\nu =
\int_{\R^n\times\R^n}\phi(x,X(x))\rho(x)\,dx.
\end{equation*}
The measure $\nu$ is supported in $\Gamma$ and, being the pushforward of a probability, it is a probability as well. 

\smallskip
\noindent
\textbf{General case.}
Let $\mu$ be an arbitrary closed probability measure on $\R^n$. Let $L>0$ be such that if $(x,v)\in\supp\mu$ then $\|v\|\leq L$. Let $U$ be a bounded, open subset of $\R^n$ that contains $\supp\mu$ and satisfies $\dist(\supp\mu,\partial U)\geq 1$.
Let $\psi\colon\R^{2n}\to\R$ be a mollifier, that is, a $C^\infty$, compactly supported, radially
symmetric $\psi(x)=\psi(\|x\|)$, nonnegative function with $\int_{\R^n}\psi=1$ and $\supp\psi\subseteq B_1(0)\subset\R^n$, and let $\psi_r(x)=r^{-2n}\psi(x/r)$ for $0<r<1$. 
The probability measure $\psi_r*\mu$ is smooth and compactly supported; in fact,
\[\supp\psi_r*\mu\subset U\times\{v\in\R^n:\|v\|\leq L+r\}.\]
Denote $\bar v^\eta$ the centroid field of the measure $\eta$. Then the centroid field and the projected densities of the convolution, $\bar v^{\psi_r*\mu}$ and $\pi_*(\psi_r*\mu)$, are smooth and converge to $\bar v^\mu$ and $\pi_*\mu$, respectively, as $r\searrow 0$.

Analogously to the definition of $\Gamma$ in the smooth case, let $\Gamma_r$ be the subset of $\lipschitz_{L+r}$ that consists of all flow lines of $\bar v^{\psi_r*\mu}$ that are defined on all of $\R$, and let
\[\nu_r=(\operatorname{ev}|_{\Gamma_r}^{-1})_*( \pi_*(\psi_r*\mu)(x)\lebesgue_{\R^n}(x))\]
so that, for measurable $\phi\colon\R^n\times\R^n\to\R$,
\[\int_{\lip}\phi(\gamma(0),\gamma'(0))\,d\nu_r(\gamma)=\int_{\R^n\times\R^n}\phi(x,\overline v_x^{\psi_r*\mu})\,d(\pi_*(\psi_r*\mu))(x).\]
The probability measure $\nu_r$ is supported in the set $\lipschitz_{L+r}$. 

The set $\lipschitz_{L+1}$, which contains $\lipschitz_{L+r}$ for $0\leq r<1$, is sequentially compact. Indeed, if we have a family $\{\gamma_i\}_{i\in I}\subset \lipschitz_{L+1}$, then it is equibounded (as the image of each curve is contained in the bounded set $U$) and equicontinuous (because all its members have Lipschitz constant at most $L+1$), so by the Arzel\`a--Ascoli theorem we can extract a subsequence $\{\gamma_i^1\}_{i\in \N}$ that converges in the interval $[-1,1]$. We then produce, by induction, a sequence of subsequences: assuming we already extracted a subsequence $\{\gamma^j_i\}_i$ that converges in $[-j,j]$, the Arzel\`a--Ascoli theorem tells us that there is a further subsequence $\{\gamma^{j+1}_i\}_i\subseteq \{\gamma^j_i\}_i$ of curves that converge in the interval $[-j-1,j+1]$. We then pick the diagonal sequence $\{\gamma^i_i\}_i$, which converges throughout $\R$ to a curve in $\lipschitz_{L+1}$.

Since it is also metrizable with $\dist(\gamma_1,\gamma_2)=\|\gamma_1-\gamma_2\|_{\infty}$, $\lipschitz_{L+1}$ is also compact. Prokhorov's theorem \cite{prokhorov} implies that there is a weakly convergent sequence $\{\nu_{r_i}\}_i\subset\lipschitz_{L+1}$ with $r_i\searrow0$. It is then a routine procedure to check that the limit probability measure $\nu=\lim_i\nu_{r_i}$ satisfies \eqref{eq:superposition}.

\medskip

\noindent
{\bf Acknowledgements.} The authors acknowledge the support of ANR-3IA Artificial and Natural Intelligence Toulouse Institute. JB and EP also thank Air Force Office of Scientific Research, Air Force Material Command, USAF, under grant numbers FA9550-19-1-7026, FA9550-18-1-0226, and ANR MasDol. JB acknowledges the support of ANR Chess, grant ANR-17-EURE-0010 and ANR OMS.

 \bibliography{bib}{}
 \bibliographystyle{plain}

\end{document}